\documentclass[11pt]{article}
\usepackage[margin=1in]{geometry}
\usepackage{graphicx,xcolor,cite}
\usepackage{amssymb,amsmath,amsthm,mathrsfs,paralist,bm,esint,setspace}
\RequirePackage[colorlinks,citecolor=blue,urlcolor=blue]{hyperref}

\newcommand{\R}{{\mathbb{R}}}
\newcommand{\E}{\mathrm{E}}
\newcommand{\HH}{\mathcal{H}}
\renewcommand{\P}{\mathrm{P}}
\renewcommand{\d}{\mathrm{d}}
\newcommand{\e}{\mathrm{e}}
\newcommand{\Var}{\text{\rm Var}}

\DeclareMathOperator{\Cov}{\text{\rm Cov}}

\input cyracc.def
\font\tencyi=wncyi10

\title{Spatial ergodicity for SPDEs via Poincar\'e-type inequalities\thanks{%
	Research supported in part by  NSF grants DMS-1811181 (D.N.) and DMS-1855439 (D.K.).}}
\author{Le Chen\\University of Nevada, Las Vegas\\\texttt{le.chen@unlv.edu}\\
	\and
		Davar Khoshnevisan\\University of Utah\\\texttt{davar@math.utah.edu}\\
	\and\and
		David Nualart\\University of Kansas\\\texttt{nualart@ku.edu}\\
	\and
		Fei Pu\\University of Utah\\\texttt{pu@math.utah.edu}
	}
\date{July 24, 2019}
\begin{document}
\newtheorem{stat}{Statement}[section]
\newtheorem{proposition}[stat]{Proposition}
\newtheorem*{prop}{Proposition}
\newtheorem{corollary}[stat]{Corollary}
\newtheorem{theorem}[stat]{Theorem}
\newtheorem{lemma}[stat]{Lemma}
\theoremstyle{definition}
\newtheorem{definition}[stat]{Definition}
\newtheorem*{cremark}{Remark}
\newtheorem{remark}[stat]{Remark}
\newtheorem*{OP}{Open Problem}
\newtheorem{example}[stat]{Example}
\newtheorem{nota}[stat]{Notation}
\numberwithin{equation}{section}
\maketitle

\begin{abstract}
	Consider a parabolic stochastic PDE of the form
	$\partial_t u=\frac12\Delta u + \sigma(u)\eta$, where
	$u=u(t\,,x)$ for $t\ge0$ and $x\in\R^d$, $\sigma:\R\to\R$
	is Lipschitz continuous and non random, and $\eta$ is a centered Gaussian
	noise that is white in time and colored in space, with a possibly-signed
	homogeneous spatial correlation $f$. If, in addition,
	$u(0)\equiv1$, then we prove that, under a mild
	decay condition on $f$, the process $x\mapsto u(t\,,x)$ is stationary and ergodic at
	all times $t>0$.  It has been argued that, when coupled with
	moment estimates,  spatial ergodicity of $u$
	teaches us about the intermittent nature of the solution to such
	SPDEs \cite{BertiniCancrini1995,KhCBMS}.
	Our results provide rigorous justification
	of  such discussions.
	
	Our methods hinge on novel facts from harmonic analysis and functions
	of positive type, as well as from Malliavin calculus and Poincar\'e inequalities.
	We further showcase the utility of these Poincar\'e inequalities
	by: (a) describing conditions that ensure that the random field $u(t)$
	is mixing for every $t>0$; and by (b) giving a quick proof of a conjecture of Conus et al
	\cite{CJK12} about the ``size'' of the intermittency islands of $u$.
	
	The ergodicity and the mixing results of this paper are sharp, as they include the classical theory
	of Maruyama \cite{Maruyama} (see also Dym and McKean \cite{DymMcKean})
	in the simple setting where the nonlinear term $\sigma$ is a constant function.
\end{abstract}

\bigskip\bigskip


	\\
	
\noindent{\it Running head:} Spatial ergodicity for SPDEs.

\newpage
{
  \hypersetup{linkcolor=black}
  \tableofcontents
}
% \tableofcontents

\section{Introduction}

The principal aim of this article is to establish relatively simple-to-check, but also
broad, conditions under which the solution
$u=\{u(t\,,x)\}_{t\ge0\,,x\in\R^d}$ to a parabolic stochastic PDE is
\emph{spatially stationary and ergodic}.
Equivalently, we would like to know conditions under which
$u(t)$ is stationary and ergodic, in its spatial variable $x$, at all times $t>0$. This problem,
and its relation to intermittency, have been
mentioned informally  for example
in the introduction of Bertini and Cancrini \cite{BertiniCancrini1995} (see also
\cite[Ch.\ 7]{KhCBMS}). This problem is also connected somewhat loosely to novel
applications of Malliavin calculus to central limit theorems
for parabolic SPDEs; see Huang et al \cite{HuangNualartViitasaari2018,HuangNualartViitasaariZheng2019}.

In order for spatial ergodicity to be a meaningful property, one needs to consider parabolic SPDEs
for which the solution is {\it a priori} a stationary process in its spatial variable. Thus, we study the
following archetypal parabolic problem:
\begin{equation}\label{SHE}
\begin{cases}
\displaystyle
\partial_t  u = \frac12\Delta u + \sigma(u)\eta
		&\text{on $(0\,,\infty)\times\R^d$},\\[0.5em]
\displaystyle	u(0)\equiv 1,
\end{cases}\end{equation}
where $\sigma\not\equiv0$ is Lipschitz continuous and non random, and $\eta$
denotes a generalized, centered, Gaussian random field with
covariance form
\[
	\E\left[\eta(t\,,x)\eta(s\,,y)\right] = \delta_0(t-s) f(x-y)
	\qquad\text{for all $s,t\ge0$ and $x,y\in\R^d$},
\]
and where $f$ is a nonnegative-definite distribution  on $\R^d$.
Somewhat more formally, the Wiener-integral process
$\psi \mapsto\eta(\psi) :=\int_{\R_+\times\R^d}
 \psi(t\,,x)\,\eta(\d t\,\d x)$ is linear a.s.\ and satisfies
\begin{equation}\label{Cov:eta}
	\text{\rm Cov}\left( \eta( \psi _1)\,,\eta( \psi _2)\right) = \int_0^\infty
	\left\langle  \psi_1(t)\,, \psi_2(t)*f\right\rangle_{L^2(\R^d)}\d t,
\end{equation}
for every $\psi_1,\psi_2$ in
the space $C_c(\R_+\times\R^d)$ of all compactly-supported, continuous,
real-valued functions on $\R_+\times\R^d$.

The solution theory for \eqref{SHE} is particularly well established when
the spatial correlation $f$ of the noise $\eta$
belongs to the space $\mathfrak{M}_+(\R^d)$
of all nonnegative-definite tempered Borel measures on $\R^d$. In that
case, it is well known that the Fourier transform is a 1-1 linear mapping from
$\mathfrak{M}_+(\R^d)$ to itself. That is, $f\in\mathfrak{M}_+(\R^d)$
if and only if $\hat{f}\in\mathfrak{M}_+(\R^d)$,\footnote{%
	Here and throughout,
	we normalize the Fourier transform so that
	$\hat{\psi}(x) = \int_{\R^d}\exp(ix\cdot y) \psi(y)\,\d y$ for all $\psi\in L^1(\R^d)$
	and $x\in\R^d$.}
and
\begin{equation}\label{Fourier_transform}
	\int_{\R^d}\phi\,\d f
	= \int_{\R^d} \overline{\hat{\phi}}\,\d \hat{f}
	= \int_{\R^d} \hat{\phi}\,\d \hat{f}
	\qquad\text{for all $\phi\in\mathscr{S}(\R^d)$,}
\end{equation}
where $\mathscr{S}(\R^d)$ denotes the usual space of all test functions
of rapid decrease on $\R^d$.
The measure $\hat{f}$ is sometimes called the \emph{spectral measure}
of $f$. And the theory of Dalang \cite{Dalang1999} implies that if
\begin{equation}\label{Dalang}
	f\in\mathfrak{M}_+(\R^d),
	\quad\text{and}\quad
	\int_{\R^d}\frac{\hat{f}(\d x)}{\lambda+\|x\|^2}<\infty
	\quad\text{for one --- hence all --- $\lambda>0$},
\end{equation}
then \eqref{SHE} has a random-field solution $u$ that is unique subject to
the following integrability condition:
\begin{equation}\label{cond:moment}
	\sup_{t\in[0,T]}\sup_{x\in\R^d}\E\left(|u(t\,,x)|^k\right)<\infty
	\qquad\text{for every $T>0$ and $k\ge2$}.
\end{equation}
Moreover, $(t\,,x)\mapsto u(t\,,x)$ is both $L^k(\Omega)$-continuous and
weakly stationary in $x$ for every $t>0$.
Furthermore, it is known that Condition \eqref{Dalang} is
necessary and sufficient for example when $\sigma$ is a non-zero constant;
see Dalang \cite{Dalang1999}, as well as Peszat and Zabczyk \cite{PeszatZabczyk2000}.

Let
\begin{equation}\label{ball}
	\mathbb{B}_r := \left\{ x\in\R^d:\, \|x\|\le r\right\}\qquad\text{for every $r>0$}.
\end{equation}
Our first result is a detailed description of the spatial ergodicity of $f$
in the case that $f$ satisfies Dalang's condition \eqref{Dalang}.

\begin{theorem}\label{th:1}
	If $f$ satisfies  \eqref{Dalang}, then $u(t)$ is a stationary random field for every $t>0$.
	Moreover, the following are valid:
	\begin{compactenum}
		\item[1.] If $\hat{f}\{0\}=0$, then $u$ is spatially ergodic;
		\item[2.] $\hat{f}\{0\}=0$ iff $f(\mathbb{B}_r)=o(r^d)$ as $r\to\infty$;
		\item[3.] If $\sigma$ is a nonzero constant, then $\hat{f}\{0\}=0$
			iff $u$ is spatially ergodic;
		\item[4.] $\hat{f}\{0\}>0$ iff $\hat{f}$ has an atom.
	\end{compactenum}
\end{theorem}

If $f$ is a \emph{function} that satisfies \eqref{Dalang}, then
part 2 of Theorem \ref{th:1} can be recast as follows:
\begin{equation}\label{HfH}
	\hat{f}\{0\}=0
	\quad\text{if and only if}\quad
	\lim_{r\to\infty} \frac{1}{|\mathbb{B}_r|}\int_{\mathbb{B}_r} f(x)\,\d x=0.
\end{equation}
Thus, we see that when $f$ is a function, $\hat{f}\{0\}=0$ iff the asymptotic average of $f$ is zero.

\begin{remark}
	Maruyama \cite{Maruyama} has shown that if a
	1-parameter, stationary Gaussian process $X=\{X_t\}_{t\in\R}$
	has a continuous covariance function $\varrho$, then $X$ is
	ergodic if and only if $\hat{\varrho}$ has no atoms;
	see also Dym and McKean \cite[\S3.9]{DymMcKean}.
	When $d=1$, Part 3 of Theorem \ref{th:1} can be derived easily
	by verifying Maruyama's  condition, using part 4 of Theorem \ref{th:1};
	part 4 of Theorem \ref{th:1} and its connection to mean-zero
	property \eqref{HfH} of $f$ appear to be new, at this level of generality,
	even in the Gaussian case.
\end{remark}

There is also a literature on well-posedness and regularity theory for \eqref{SHE}
when $f$ is a distribution that is
not necessarily in $\mathfrak{M}_+(\R^d)$, though such results tend to be
applicable in a more specialized setting as compared with the
theory of Dalang \cite{Dalang1999}; see for example
\cite{ChenHuNualart2017,ChenHuang2019,%
HuHuangLeNualartTindel2017,HuHuangNualartTindel2015,%
HuangLeNualart2017}. Henceforth, we consider the case that $f$ is
a nonnegative-definite, but possibly signed, \emph{function} of the form,
\begin{equation}\label{f=h*h}
	f = h*\tilde{h},
\end{equation}
where $h:\R^d\to\R$ has enough regularity to ensure among other things
that the  convolution in \eqref{f=h*h} is well defined, and $\tilde{h}(x):=h(-x)$ defines
the reflection of $h$.
In this case, \eqref{Cov:eta} is equivalent to the elegant formula,
\[
	\Var\left( \eta( \psi )\right) = \int_0^\infty
	\left\| \psi(t)*h\right\|_{L^2(\R^d)}^2\d t,
	\qquad\text{
	valid for all $\psi\in C_c(\R_+\times\R^d)$.}
\]

In this context, we prove that a mild integrability condition on $h$ implies that $|h|\in H_{-1}(\R^d)$
--- see \eqref{cond:omega} and Lemma \ref{lem:Dalang:Lp} ---
which in turn implies the existence of a
spatially stationary random-field solution $u$ to \eqref{SHE} that is unique
subject to \eqref{cond:moment}; see Theorem \ref{th:exist}.
More significantly, we prove that the ensuing Condition \eqref{cond:omega}
on $h$ ensures that $u$
is spatially ergodic.

In any case, the end result is the following theorem.\footnote{%
	For a very brief
	discussion of relevant measurability issues, see Remark \ref{rem:doob} below.}
In order to present that result we first recall \eqref{ball}, and
then define for every $r>0$,
\begin{equation}\label{eq:omega}
	\omega_d(r) := \begin{cases}
		1&\text{if $d=1$},\\
		r\log_+(1/r)&\text{if $d=2$},\\
		r&\text{if $d\ge3$,}
	\end{cases}
\end{equation}
where $\log_+(z):=\log(z\vee\e)$ for all $z\in\R$.

\begin{theorem}\label{th:main:intro}
	Assume that the spatial correlation function $f:\R^d\to\R$ satisfies \eqref{f=h*h},
	where $h\in L^p_{\text{\it loc}}(\R^d)$ for some number $p>1$, and satisfies
	\begin{equation}\label{cond:omega}
		\int_0^1 \left( \|h\|_{L^p(\mathbb{B}_r)} \|h\|_{L^q(\mathbb{B}_r^c)}
		+ \|h\|_{L^2(\mathbb{B}_r^c)}^2\right)\omega_d(r)\,\d r<\infty
		\quad\text{with}\quad
		\text{$q:=\frac{p}{p-1}.$}
	\end{equation}
	Then, the SPDE \eqref{SHE} has a spatially stationary and ergodic random-field solution
	$u$ that is unique subject to the integrability condition \eqref{cond:moment}.
\end{theorem}

\begin{remark}
	In order to be concrete, we have selected the initial data to be identically 1 in Theorems
	\ref{th:1} and \ref{th:main:intro}.
	The same arguments show that Theorems \ref{th:1} and \ref{th:main:intro}
	continue to hold when the initial data is an arbitrary
	stationary random field $\{u(0\,,x)\}_{x\in\R^d}$ that is independent of $\eta$
	and is continuous in $L^k(\Omega)$ for every real number $k\ge1$.
\end{remark}

In the case $f$ is signed and satisfies \eqref{f=h*h}, the following
presents easy-to-check conditions on $h$ in order for \eqref{SHE} to have a
unique random-field solution that is spatially
ergodic (as well as stationary).

\begin{corollary}\label{co:main:intro}
	Suppose that $h:\R^d\to\R$ is Borel measurable, and either that $h\in L^2(\R^d)$ or
	that there exist $\alpha\in(0\,,d\wedge 2)$
	and $\beta>0$ such that
	\begin{equation}\label{cond:co:main:intro}
		\sup_{\|w\|<1} \|w\|^{(d+\alpha)/2}|h(w)| <\infty
		\quad\text{and}\quad
		\sup_{\|z\|>1} \|z\|^{(d+\beta)/2}|h(z)| <\infty.
	\end{equation}
	Then, \eqref{SHE} has a random-field solution $u$ that is unique
	subject to the moment condition \eqref{cond:moment}.
	Moreover, $u(t)$ is stationary and ergodic for every $t>0$.
\end{corollary}

It is worth noting that, whereas \eqref{cond:omega} is a global integrability condition on $h$,
\eqref{cond:co:main:intro} involves: (i) A local condition on the behavior of $h$
near the origin; and (ii) A separate local-at-infinity (growth) condition on $h$. We will
show quickly in \S\ref{sec:Poincare} that
\eqref{cond:co:main:intro} implies \eqref{cond:omega}.

It is also worth noting that the first
(local) condition on $h$ in \eqref{cond:co:main:intro} is there merely to ensure that
$|h|\in H_{-1}(\R^d)$, which in turn will imply that
\eqref{SHE} has a solution. The second (growth)
condition on $h$ in \eqref{cond:co:main:intro}  is the more interesting hypothesis.
That condition is responsible
for ensuring that $h$ --- whence also $f$ --- decays sufficiently rapidly so that
spatial ergodicity of the solution $u$ to \eqref{SHE} is ensured.

\bigskip
Our ergodicity results (Theorems \ref{th:1} and \ref{th:main:intro})
are consequences of the following two
Poincar\'e-type inequalities.

\begin{theorem}[Poincar\'e inequality I]\label{th:Poincare}
	Assume that the conditions of Theorem \ref{th:1} are met. Then, for
	every number $T>0$ there exists a real number $C>0$ such that
	\begin{equation}\label{eq:Poincare}
		\sup_{t\in[0,T]}\Var\left(\frac{1}{N^d}\int_{[0,N]^d} \prod_{j=1}^k g_j
		(u(t\,,x + \zeta^j))\,\d x\right) \le
		\frac{Ck^2}{N^d} f\left( [-N\,,N]^d\right),\\
	\end{equation}
	uniformly for every integer $k\ge1$, real number $N>1$,  $\zeta^1,\ldots,\zeta^k\in\R^d$,
	and all Lipschitz-continuous functions
	$g_1,\ldots,g_k:\R\to\R$ that satisfy
	    \begin{equation}\label{eq:WLOG}
		g_j(0)=0\quad\text{and}\quad
		\text{\rm Lip}(g_j)=1
		\qquad \text{for every $j=1,\ldots,k$}.
	\end{equation}
\end{theorem}

It is perhaps needless to add that
we are using the standard convention that $\text{\rm Lip}(\psi)$ denotes
the Lipschitz constant of $\psi:\R^d\to\R$; that is,
\[
	\text{\rm Lip}(\psi) := \sup_{x\neq y}\frac{%
	|\psi(x)-\psi(y)|}{\|x-y\|}.
\]
\begin{theorem}[Poincar\'e inequality II]\label{th:Poincare:bis}
	Assume that the conditions of Theorem \ref{th:main:intro} are met. Then, for every number $T>0$ there exists a real number $C>0$ such that
	\begin{equation}\label{eq:Poincare:bis}
		\sup_{t\in[0,T]}\Var\left(\frac{1}{N^d}\int_{[0,N]^d} \prod_{j=1}^k g_j
		(u(t\,,x + \zeta^j))\,\d x\right) \le
		\frac{Ck^2}{N^d}\int_{[-N,N]^d} \left( |h|*|\tilde{h}|\right)(x)\,\d x,\\
	\end{equation}
	uniformly for every integer $k\ge1$, real number $N>1$,  $\zeta^1,\ldots,\zeta^k\in\R^d$,
	and all Lipschitz-continuous functions
	$g_1,\ldots,g_k:\R\to\R$ that satisfy \eqref{eq:WLOG}.
\end{theorem}

Theorems \ref{th:Poincare} and
\ref{th:Poincare:bis} are proved in \S\ref{sec:Poincare}. The proofs make a novel appeal
to the Malliavin calculus, specifically to the Clark--Ocone formula; see \S\ref{sec:Malliavin}.
Next, we would like
to explain briefly why Theorems \ref{th:Poincare} and \ref{th:Poincare:bis}
are indeed Poincar\'e-type inequalities, as is suggested also by the title of the paper.

Let $F$ denote a square-integrable random variable in a nice filtered probability space
that is rich enough to carry a sufficiently-nice Gaussian measure.  In
this context, the Poincar\'e
inequality states roughly that one can estimate the variance of $F$ by finding
good estimates for the Malliavin derivative of $F$. Capitaine, Hsu, and Ledoux \cite{CHL}
have observed that the Poincar\'e inequality can be deduced from the Clark-Ocone
formula. The argument is elegant and brief. We describe it next in the context of
1-D Brownian motion $B$.  Let us construct $B=\{B_t\}_{t\ge0}$ on the  space
$\Omega:=C(\R_+\,;\R)$ via the coordinate map
[$B_t(\omega) := \omega(t)$ for all $\omega\in\Omega$
and $t\ge0$] and Wiener measure $\mathbb{W}$.
Let $\{\mathcal{B}_t\}_{t\ge0}$ denote the natural filtration of $B$,
augmented in the usual way.
According to the Clark--Ocone formula \cite{O84} --- see also
\cite[Proposition 1.3.14]{Nualart}) --- if $F\in L^2(\Omega\,,\mathcal{B}_T\,,\mathbb{W})$ for
some $T>0$, and is in a suitable Gaussian Sobolev space, then
$F - \E_{_{\mathbb{W}}} F = \int_0^T \E_{_{\mathbb{W}}} ( D_tF \mid \mathcal{B}_t)\,\d B_t$
a.s.\ [$\mathbb{W}$],
where $DF$ denotes the Malliavin derivative of $F$.
The It\^o isometry, Jensen's inequality,
and two back-to-back appeals to Fubini's theorem, together imply that
\begin{equation}\label{CHL}
	\Var_{_{\mathbb{W}}}(F) = \E_{_{\mathbb{W}}}
	\left( \int_0^T\left|\E_{_{\mathbb{W}}} ( D_tF \mid \mathcal{B}_t)\right|^2\,\d t\right)
	\le \int_0^T\E_{_{\mathbb{W}}}
	\left(|D_tF|^2\right)\d t = \E_{_{\mathbb{W}}}
	\left( \|DF\|_{L^2[0,T]}^2\right),
\end{equation}
which is precisely the classical Poincar\'e inequality on the Wiener space
$(\Omega\,,\vee_{t\ge0}\mathcal{B}_t\,,\mathbb{W})$. This
is one way to state more carefully our earlier assertion that
good information on the magnitude of the second moment of
$\|DF\|_{L^2[0,T]}$ can imply an upper bound
on the variance of $F$. Theorems \ref{th:Poincare} and \ref{th:Poincare:bis}
are certain specializations of a more complex form of
this Poincar\'e inequality (see Proposition \ref{Clark-Ocone}), wherein the above Brownian
motion $B$ is replaced by an
infinite dimensional Brownian motion. Theorems \ref{th:Poincare} and \ref{th:Poincare:bis}
include also sharp Malliavin derivative bounds, whose derivation requires additional ideas and
extra effort.

Here is a brief outline of the paper: In \S\ref{sec:example} we present an example
which shows that we cannot expect spatial ergodicity of the solution
of \eqref{SHE} unless $f$ exhibits some sort of decay at infinity,
valid even when $\sigma$ is not constant. Section \ref{sec:FPT} includes
comments and a few harmonic-analytic results on functions of positive type.
Section \ref{sec:Dalang} discusses known results on
the well-posedness of \eqref{SHE}, and discusses how
conditions of Theorem \ref{th:main:intro} ensure among other things that the absolute
value of $h$ is in the classical space Hilbert space $H_{-1}(\R^d)$.
Section \ref{subsec:part4} contains a quick proof of the folklore fact which is part 4
of Theorem \ref{th:1}.
In \S\ref{sec:SC}
we extend the stochastic Young inequality of
Walsh integrals \cite{ConusKh2012,FoondunKhoshnevisan2009}
to the case that $f$ is possibly signed
and satisfies the conditions of Theorem \ref{th:main:intro}. It is shown
in \S\ref{sec:wellposedness} that the well-posedness of \eqref{SHE} is a ready consequence
of the mentioned stochastic Young's inequality; see Theorem \ref{th:exist}.
Methods of Malliavin calculus turn out to play a central role
in the study of the spatial ergodicity of the solution, and we present
the elements of Malliavin calculus in \S\ref{sec:Malliavin}.
The stationarity
assertion of Theorem \ref{th:main:intro} is proved next in \S\ref{sec:stat}.
Theorem \ref{th:main:intro} is proved shortly following the proof of Theorem \ref{th:Poincare},
and in a final section \ref{sec:Appl}, we use our Poincar\'e inequalities to establish tight
criteria for spatial mixing of the solution to \eqref{SHE} [\S\ref{subsec:mixing}]
and also prove a conjecture of
Conus et al \cite{CJK12} related to an ``intermittency'' property
of the solution to \eqref{SHE} in a special, though important, case [\S\ref{subsec:interm}].

Let us close the Introduction with a brief description of the notation of this paper.
Throughout we write ``$g_1(x)\lesssim g_2(x)$ for all $x\in X$'' when there exists a real number
$L$ such that $g_1(x)\le Lg_2(x)$ for all $x\in X$. Alternatively, we might write
``$g_2(x)\gtrsim g_1(x)$
for all $x\in X$.'' By ``$g_1(x)\asymp g_2(x)$ for all $x\in X$'' we mean that
$g_1(x)\lesssim g_2(x)$ for all $x\in X$ and $g_2(x)\lesssim g_1(x)$ for all $x\in X$.
Finally,
``$g_1(x)\propto g_2(x)$ for all $x\in X$'' means that there exists a real number $L$
such that $g_1(x)=L g_2(x)$ for all $x\in X$.

Throughout, we write
\[
	\fint_E \psi(x)\,\d x := \frac{1}{|E|}\int_E\psi(x)\,\d x,
\]
whenever $\psi:\R^d\to\R$ is integrable on a Lebesgue-measurable set
$E\subset\R^d$ whose Lebesgue measure $|E|$ is strictly positive.
We will  use the notation,
$\|X\|_k := \{\E(|X|^k)\}^{1/k}$, valid for every real-valued random
variable $X\in L^k(\Omega)$
and every real number $k\ge1$.

\section{A non-ergodic example}\label{sec:example}
In the Introduction we alluded that
if the tails of the spatial correlation
$f$ do not vanish, then we cannot generally expect $u(t)$ to be  ergodic
for all $t\ge0$.
We now describe this in the context of
an  example in which the spatial correlation function
$f(x)$ does not decay as $\|x\|\to\infty$, the solution $u$
exists and is non-degenerate, and $u$ is
not spatially ergodic at  positive times.

First, we might as well rule out trivialities by assuming that
\begin{equation}\label{sigma(1)neq 0}
	\sigma(1)\neq0.
\end{equation}
Otherwise, one can see easily that $u(t\,,x)\equiv1$; in this case,
$u(t)$ is  ergodic for all $t\ge0$, but only in a vacuous sense.

Next, let us choose and fix a number $\lambda>0$, and suppose that
\begin{equation}\label{f=1}
	f(x)=\lambda^2\qquad\text{for all $x\in\R^d$},
\end{equation}
to ensure that the tails of $f$ do not decay. In this case, it is possible
to realize the noise $\eta(\d t\,\d x)$ as $\lambda\,\d W_t\,\d x$,
where $W$ denotes a one-dimensional Brownian motion. Thus, we can
infer from \eqref{SHE} and well-known
arguments that, under
\eqref{f=1},
\begin{equation}\label{u=X}
	u(t\,,x) = X_t\qquad\text{for all $t\ge0$ and $x\in\R^d$ a.s.,}
\end{equation}
where $X$ is the unique (strong)
solution of the one-dimensional It\^o SDE,
\[
	\d X_t = \lambda\sigma(X_t)\,\d W_t,
	\qquad
	\text{subject to  $X_0=1.$}
\]
Standard estimates now reveal that
\[
	\lim_{t\to0^+}
	\frac1t\Var(X_t) =  \lambda^2\sigma^2(1),
\]
whence $\Var(X_t)>0$
for all $t$ small. Thus, we conclude from \eqref{u=X} that, under conditions
\eqref{sigma(1)neq 0} and \eqref{f=1}, the process $u$ is not  spatially ergodic.
In fact, a little more effort shows that $\Var(X_t)>0$ for all $t>0$, thanks to the Markov property.
And this implies that $u(t)$ is not ergodic for any $t>0$.

\section{Harmonic analysis}\label{sec:harmonic}
\subsection{Functions of positive type}\label{sec:FPT}

Let us recall the notation from \eqref{ball} for closed centered balls,
and recall the following from classical harmonic analysis \cite{Kahane}:

\begin{definition}\label{def:PT}
	We say that a function $g:\R^d\to\R$ is of \emph{positive type} if:
	\begin{compactenum}
		\item $g$ is locally integrable and nonnegative definite in the sense of distributions
			(that is, $\hat{g}\ge0$ and hence a Borel measure, thanks to
			the Riesz representation theorem);
		\item The restriction of $g$ to $\mathbb{B}_r^c$ is a uniformly continuous (and
			hence also bounded) function for every $r>0$; and
		\item $\lim_{\|x\|\to\infty}g(x)=0$.\footnote{Some authors insist
			that $g$ is of positive type if, in addition to the requirements of Definition \ref{def:PT},
			$g(0):=\lim_{x\to0}g(x)=\infty$. Others do not insist that $g$ vanishes at
			infinity.}
	\end{compactenum}
\end{definition}

Typical examples include $g(x)=\exp(-\alpha\|x\|^\beta)$ and/or
$g(x)=(\alpha'+\|x\|^\beta)^{-1}$,
for constants
$\alpha\ge0$, $\alpha'>0$, and $\beta\in(0\,,2]$, etc. There
are also unbounded examples such as Riesz kernels
($g(x)=\|x\|^{-\gamma}$ for $\gamma\in(0\,,d)$), as well as products of
the preceding such as $g(x)=\|x\|^{-\gamma}\exp(-\alpha\|x\|^\beta)$, etc.

The main goal of this section is to present a  family $\cup_{p>1}\mathcal{F}_p(\R^d)$
of real-valued functions on $\R^d$
that can be used explicitly to construct a large number
of functions of positive type that are central to our analysis.
We will also use this opportunity to
introduce another vector space $\cup_{p>1}\mathcal{G}_p(\R^d)$ of functions
that will play a prominent role in later sections (though not in this one).

\begin{definition}\label{def:F_p}
	Choose and fix a real number $p>1$, and
	define $\mathcal{F}_p(\R^d)$ to be the collection of all $h \in L^p_{\text{\it loc}}(\R^d)$
	that satisfy
	\begin{equation}\label{cond:th:main:intro}
		\int_0^1 s^{d-1} \left( \| h \|_{L^p(\mathbb{B}_s)}
		 \| h \|_{L^q(\mathbb{B}_s^c)} + \| h \|_{L^2(\mathbb{B}_s^c)}^2\right)\d s<\infty
		 \quad\text{with}\quad \text{$q:=\frac{p}{p-1}$.}
	\end{equation}
	We also define $\mathcal{G}_p(\R^d)$ to be the collection of all
	functions $h\in L^p_{\text{\it loc}}(\R^d)$ that satisfy \eqref{cond:omega}.
\end{definition}

In this section we study some of the basic properties of the elements of
the vector spaces
$\cup_{p>1}\mathcal{F}_p(\R^d)$ and $\cup_{p>1}\mathcal{G}_p(\R^d)$.
It might help to add that, notationally speaking, the functions $h$ in $\cup_{p>1}
\mathcal{G}_p(\R^d)$ and $\cup_{p>1}\mathcal{F}_p(\R^d)$ will be potential
candidates for the function $h$ in \eqref{f=h*h}, which are then used to form
the spatial correlation function $f$ in \eqref{SHE}. Thus, the notation should aid
the reading, and not hinder it.

\begin{lemma}\label{lem:F_p}
	The following are valid for every $p>1$, where $q:=p/(p-1)$:
	\begin{compactenum}
		\item $\mathcal{G}_p(\R^d)\subseteq
			\mathcal{F}_p(\R^d)\subseteq
			L^1_{\text{\it loc}}(\R^d)$ for all $d\ge1$, and
			$\mathcal{G}_p(\R)=\mathcal{F}_p(\R)$.
		\item $\|h\|_{L^p(\mathbb{B}_r)}$,
			$\|h\|_{L^q(\mathbb{B}_r^c)}$, and
			$\|h\|_{L^2(\mathbb{B}_r^c)}$ are  finite for every
			$h\in\mathcal{F}_p(\R^d)$ and $r>0$.
		\item If $h\in\mathcal{F}_p(\R^d)$, then
			\begin{equation}\label{eq:F_p}
				\int_0^r s^{d-1} \left( \| h \|_{L^p(\mathbb{B}_s)}
				 \| h \|_{L^q(\mathbb{B}_s^c)} +
				 \| h \|_{L^2(\mathbb{B}_s^c)}^2\right)\d s<\infty
				 \qquad\text{for every $r>0$.}
			\end{equation}
		\item If $h\in\mathcal{G}_p(\R^d)$, then
			\begin{equation}\label{eq:G_p}
				\int_0^r \left( \| h \|_{L^p(\mathbb{B}_s)}
				 \| h \|_{L^q(\mathbb{B}_s^c)} +
				 \| h \|_{L^2(\mathbb{B}_s^c)}^2\right)\omega_d(s)\,\d s<\infty
				 \qquad\text{for every $r>0$.}
			\end{equation}
	\end{compactenum}
\end{lemma}

\begin{proof}
	We have $\mathcal{G}_p(\R^d)\subset\mathcal{F}_p(\R^d)$ for all $d\ge2$
	and $\mathcal{G}_p(\R)=\mathcal{F}_p(\R)$ because
	of \eqref{eq:omega}; and the local integrability of $h\in\mathcal{F}_p(\R^d)$
	is a consequence of H\"older's inequality. This proves part 1.
	We concentrate on the remaining assertions of the lemma.
	
	First, let us note that if $p>1$ and $h \in\mathcal{F}_p(\R^d)$, then $h$ is locally in $L^p(\R^d)$
	and hence $\| h \|_{L^p(\mathbb{B}_r)}$ is finite for every $r>0$. In particular,
	\begin{equation}\label{3:phi}
		\| h \|_{L^q(\mathbb{B}_r^c)} + \| h \|_{L^2(\mathbb{B}_r^c)}<\infty,
	\end{equation}
	for almost every $r\in[0\,,1]$. Since both of the norms in \eqref{3:phi}
	are monotonically-decreasing functions of $r$,
	it follows that in fact \eqref{3:phi} holds for every $r>0$. This proves  part 2 of
	the lemma.
	
	Next, suppose $r>1$ and observe that
	\[
		\int_1^rs^{d-1} \left( \| h \|_{L^p(\mathbb{B}_s)}
		 \| h \|_{L^q(\mathbb{B}_s^c)} +
		 \| h \|_{L^2(\mathbb{B}_s^c)}^2\right)\d s
		 \le \left(\|h\|_{L^p(\mathbb{B}_r)}\|h\|_{L^q(\mathbb{B}_1^c)}
		 + \|h\|_{L^2(\mathbb{B}_1^c)}^2\right)\left(\frac{r^{d}-1}{d}\right)
	\]
	is finite. This and the definition of the vector space $\mathcal{F}_p(\R^d)$
	together imply that \eqref{eq:F_p} holds; \eqref{eq:G_p} is proved similarly.
\end{proof}

It follows from local integrability that the Fourier transform of every
function $ h \in\mathcal{F}_p(\R^d)$
($p>1$) is a well-defined distribution. In particular,  both $f= h *\tilde{ h }$ and
$| h |*|\tilde{ h }|$ are also well-defined distributions.
Of course, all such distributions are nonnegative-definite as well. The following shows that both
$h*\tilde{h}$ and $|h|*\tilde{|h|}$ are in fact fairly
nice nonnegative-definite functions from $\R^d$ to the extended real numbers $\R\cup\{\infty\}$.
	
\begin{proposition}\label{pr:PD}
	If $ h \in\mathcal{F}_p(\R^d)$ for some $p>1$, then
	$ h *\tilde{ h }$ and $| h |*|\tilde{ h }|$ are functions of positive type. Moreover,
	for every $r>0$,
	\begin{equation}\label{eq:PD}
		\sup_{\|x\|>2r} \left| \left(  h * \tilde{ h }\right)(x)\right| \le
		\sup_{\|x\|>2r} \left( | h |*|\tilde{ h }|\right)(x) \le
		2 \| h \|_{L^p(\mathbb{B}_r)}
		 \| h \|_{L^q(\mathbb{B}_r^c)} + \| h \|_{L^2(\mathbb{B}_r^c)}^2,
	\end{equation}
	and
	\begin{equation}\label{eq:PD:int}
		\int_{\mathbb{B}_r}\left( | h |*|\tilde{ h }|\right)(x)\,\d x
		\lesssim\int_0^{2r} s^{d-1}\left( \| h \|_{L^p(\mathbb{B}_s)}
		 \| h \|_{L^q(\mathbb{B}_s^c)} +
		 \| h \|_{L^2(\mathbb{B}_s^c)}^2\right)\d s,
	\end{equation}
    where the implied constant depends only on $d$, and $q:= p/(p-1)$.
\end{proposition}

\begin{proof}
	The argument  hinges loosely on old ideas that are
	motivated by the literature on potential theory of L\'evy processes; see in particular
	Hawkes \cite{Hawkes,Hawkes1984}.
	
	Let us choose and fix numbers $r,s>0$ and $x\in\R^d$ such that $\|x\|>r+s$.
	On one hand, if $y\in \mathbb{B}_r$ then certainly $\|x-y\|>s$, whence
	\[
		\int_{\mathbb{B}_r} | h(y) h(y-x) | \,\d y \le
		\| h\|_{L^p(\mathbb{B}_r)}\| h\|_{L^q(\mathbb{B}_s^c)},
	\]
	by H\"older's inequality.
	On the other hand, H\"older's inequality ensures that for every $z\in\R^d$
	and $t>0$,
	\begin{align*}
		\int_{\mathbb{B}_r^c} | h(y) h (y-z) | \,\d y &\le \int_{\substack{%
			\|y\|>r\\\|z-y\|<t}}| h (y) h (y-z) |\,\d y +
			\int_{\substack{%
			\|y\|>r\\\|z-y\|>t}}| h (y) h (y-z)|\,\d y\\
		&\le \| h\|_{L^p(\mathbb{B}_t)}\| h\|_{L^q(\mathbb{B}_r^c)} +
			\| h\|_{L^2(\mathbb{B}_r^c)}\| h\|_{L^2(\mathbb{B}_t^c)}.
	\end{align*}
	
	Combine the above bounds to find that
	\begin{equation}\label{r+s}
		\sup_{\|x\|>r+s}\left( |h|*|\tilde{h}|\right)(x) \le
		\| h\|_{L^p(\mathbb{B}_r)}\| h\|_{L^q(\mathbb{B}_s^c)} +
		\| h\|_{L^p(\mathbb{B}_t)}\| h\|_{L^q(\mathbb{B}_r^c)} +
		\| h\|_{L^2(\mathbb{B}_r^c)}\| h\|_{L^2(\mathbb{B}_t^c)},
	\end{equation}
	for every $r,s,t>0$.
	If $h\in\mathcal{F}_p(\R^d)$ then certainly
	$|h|\in\mathcal{F}_p(\R^d)$ also, and we can set $s=r=t$
	in order to deduce \eqref{eq:PD} from \eqref{r+s}.
	Also, we may first let $s\to\infty$, and then let $r\to\infty$
	in \eqref{r+s} --- in this order --- to see that
	$|h|*|\tilde{h}|$ vanishes at infinity.
	
	Next, we verify \eqref{eq:PD:int} by
	merely observing that
	\[
		\int_{\mathbb{B}_r} ( |h| *|\tilde{ h }|)(x)\,\d x
		\le \int_{\mathbb{B}_r}\Phi(\|x\|/2)\,\d x \propto
		\int_0^{2r}\Phi(s) s^{d-1}\,\d s,
	\]
	where $\Phi(t) := \sup_{\|x\|>2t}(|h| * |\tilde{ h }|)(x)$ for every $t>0$.
	Apply the already-proved part of the lemma,
	together with Lemma \ref{lem:F_p}, in order
	to see that $|h|*|\tilde{h}|\in L^1_{\text{\it loc}}(\R^d)$.
	
	The same argument that led to \eqref{r+s} (with $r=s=t$) yields that
	\[
		\sup_{\|x\|>2r}\left( |h_1|*|\tilde{h}_2|\right)(x)
		\le \|h_1\|_{L^p(\mathbb{B}_r)}\|h_2\|_{L^q(\mathbb{B}_r^c)}
		+\|h_2\|_{L^p(\mathbb{B}_r)}\|h_1\|_{L^q(\mathbb{B}_r^c)}
		+\|h_1\|_{L^2(\mathbb{B}_r^c)}\|h_2\|_{L^2(\mathbb{B}_r^c)},
	\]
	whenever $h_1,h_2\in\mathcal{F}_p(\R^d)$.
	Choose and fix an approximation
	to the identity $\{\varphi_\varepsilon\}_{\varepsilon>0}$ such that
	$\varphi_\varepsilon\in C^\infty_c(\R^d)$ for every
	$\varepsilon>0$.
	We may apply the preceding displayed inequality, once with $(h_1\,,h_2)=(h\,,h -(\varphi_\varepsilon*h))$
	and once with $(h_1\,,h_2)=(|h|\,,|h| -(\varphi_\varepsilon*|h|))$, in order to see that
	as $\varepsilon\downarrow0$,
	$(\varphi_\varepsilon* |h|*|\tilde{h}|)(x)
	\to (|h|*|\tilde{h}|)(x)$ and	
	$(\varphi_\varepsilon* h*\tilde{h})(x)
	\to (h*\tilde{h})(x)$, both valid uniformly for all $x\in\R^d$ that satisfy $\|x\|>2r$. This
	uses only the classical fact that
	\[
		\lim_{\varepsilon\downarrow0}
		\left( \left\| g - (\varphi_\varepsilon*g) \right\|_{L^p(\mathbb{B}_r)}
		+ \left\| g - (\varphi_\varepsilon*g) \right\|_{L^q(\mathbb{B}_r^c)}
		+ \left\| g - (\varphi_\varepsilon*g) \right\|_{L^2(\mathbb{B}_r^c)}
		\right)=0,
	\]
	for either $g=h$ or $g=|h|$ (see Stein \cite{Stein}), and readily implies the uniform
	continuity and boundedness of
	$h*\tilde{h}$ and $|h|*|\tilde{h}|$ off $\mathbb{B}_r$ for arbitrary $r>0$.
	This completes the proof that $h*\tilde{h}$ and $|h|*|\tilde{h}|$ are functions of positive type.
\end{proof}

\subsection{On Condition (\ref{Dalang})}\label{sec:Dalang}

As was mentioned in the Introduction, it was shown by Dalang \cite{Dalang1999}
that when $f$ is tempered and
non negative, Condition \eqref{Dalang} is an optimal sufficient condition
for the existence of a unique random-field solution to the SPDE \eqref{SHE}.
In this section, we say a few words about Dalang's Condition \eqref{Dalang}
in the setting where $f$ is a function that satisfies \eqref{f=h*h}.

First recall that the vector space $H_{-1}(\R^d)$ denotes
the completion of all rapidly-decreasing, real-valued $C^{\infty}$-functions
on $\R^d$ in the norm
\[
	\| h \|_{H_{-1}(\R^d)} :=
	\left( \int_{\R^d} \frac{|\hat{ h }(x)|^2}{1+\|x\|^2}\,\d x\right)^{1/2}.
\]
It follows immediately that $H_{-1}(\R^d)$ is Hilbertian, once endowed with the above norm and
the associated inner product,
\[
	\langle \psi_1\,, \psi_2\rangle_{H_{-1}(\R^d)} :=
	\int_{\R^d}\frac{\hat{\psi}_1(x)\overline{\hat{\psi}_2(x)}}{1+\|x\|^2}\,\d x.
\]
%This explains why the two conditions in \eqref{Dalang} are equivalent.

Next, let us define $\bm{v}_\lambda$ to be the \emph{$\lambda$-potential density}
of the heat semigroup on $\R^d$
for every $\lambda>0$. That is,
\begin{equation}\label{HRD}
	\bm{v}_\lambda(x) = \int_0^\infty\e^{-\lambda t} \bm{p}_t(x)\,\d t
	\qquad\text{for all $x\in\R^d$},
\end{equation}
where $\bm{p}$ denotes the heat kernel, defined as
\begin{equation}\label{p}
	\bm{p}_t(x) := \frac{1}{(2\pi t)^{d/2}}\exp\left(-\frac{\|x\|^2}{2t}\right)\qquad
	\text{for all $t>0$ and $x\in\R^d$}.
\end{equation}

Note that $\lambda \bm{v}_\lambda$ is  a probability density function on $\R^d$
for every $\lambda>0$.

A general theorem of Foondun and Khoshnevisan \cite{FoondunKhoshnevisan2013}
implies that when $f = h* \tilde{h}$ is a function and $h\ge0$ (and hence $f\ge0$), Dalang's condition
\eqref{Dalang} holds if and only if \footnote{In general, the proof of \eqref{FK}
	requires some effort. But, for example when $h\in L^1(\R^d)\cap L^2(\R^d)$,
	Young's inequality yields $f\in \cap_{\nu\in[1,\infty]} L^\nu(\R^d)$
	and hence \eqref{FK} is a direct consequence of Parseval's identity and the elementary facts that:
	(i) The Fourier transform of $\bm{v}_\lambda$ is
	$\hat{\bm{v}}_\lambda(z) :=\int_{\R^d}\exp\{ix\cdot z\}\bm{v}_\lambda(x)\,\d x
	= 2[2\lambda + \|z\|^2]^{-1}$ for all $z\in\R^d$; and
	(ii) $\hat{f}(z) := \int_{\R^d}\exp\{ix\cdot z\}f(x)\,\d x
		=|\hat{h}(z)|^2$ for all $z\in\R^d$.}
\begin{equation}\label{FK}
	\int_{\R^d}\bm{v}_\lambda(x)f(x)\,\d x<\infty
	\quad\text{for one, hence all, $\lambda>0$}.
\end{equation}
[To use the results of \cite{FoondunKhoshnevisan2013} we need also the easy-to-prove fact
that $h*\tilde{h}$ is lower semicontinuous in the present setting.]
An earlier result, applicable in the present context, can be found in Peszat \cite[Theorem 0.1]{Peszat2002}.

%While \eqref{Dalang} and \eqref{FK} are equivalent formulations
%of the same condition, each formulation has its technical advantages:
%On one hand, it is clear from Condition \eqref{Dalang} that if $f = h* \tilde{h}$ and $h$ is a nonnegative element of
%$H_{-1}(\R^d)$ then $\varphi*h$ is also a nonnegative element of
%$H_{-1}(\R^d)$ for every probability density $\varphi$ on $\R^d$. Thus, we see from \eqref{Dalang}
%that $H_{-1}(\R^d)$ is closed under ``smoothing.''

%On the other hand, the phrasing of Condition \eqref{FK} obviates the assertion that $H_{-1}(\R^d)$
%is closed under ``miniorization'': If $0\le h_1\le h_2$ are measurable
%and $h_2\in H_{-1}(\R^d)$, then $h_1$ is in $H_{-1}(\R^d)$ also. This
%minorization property will play a keyrole in the proof of spatial ergodicity in Theorem
%\ref{th:main:intro}.

Let us note also that if $h\ge0$ and $h\in\mathcal{F}_p(\R^d)$
for some $p>1$, then $f$ is bounded uniformly on $\mathbb{B}_r^c$ for all $r>0$.
Because in addition $\bm{v}_\lambda$ is integrable,
it follows from \eqref{FK} that, in the present setting wherein
$h\ge0$ and $h\in\cup_{p>1}\mathcal{F}_p(\R^d)$,
the harmonic-analytic condition \eqref{Dalang}---equivalently
the potential-theoretic condition \eqref{FK}---%
is equivalent to the following local version of \eqref{FK}:
\begin{equation}\label{FK:loc}
	\int_{\mathbb{B}_1}\bm{v}_\lambda(x)f(x)\,\d x<\infty
	\qquad\text{for one, hence all, $\lambda>0$.}
\end{equation}
Next, we re-interpret \eqref{FK:loc}:
It is well known, and easy to verify directly (see, for example,
\cite[Section 3.1, Chapter 10]{KhMPP}), that
\begin{equation}\label{pot:bound}
	\bm{v}_\lambda(x)\asymp \|x\|^{-d+1}\omega_d(\|x\|)
	\qquad\text{uniformly for all $x\in\mathbb{B}_1$},
\end{equation}
where $\omega_d$ was defined in \eqref{eq:omega}.
Thus, when $h\ge0$ and $h\in\cup_{p>1}\mathcal{F}_p(\R^d)$,
\begin{equation}\label{h:kappa}
	h\in H_{-1}(\R^d)\quad\text{iff}\quad
	\int_{\mathbb{B}_1} \|x\|^{-d+1}\omega_d(\|x\|)f(x)\,\d x<\infty.
\end{equation}

Next let us consider the general case where $h\in\cup_{p>1}\mathcal{F}_p(\R^d)$
is possibly signed.
Since $|f(x)|\le\big(|h|*|\tilde{h}|\big)(x)$ for all $x\in\R^d$,
we can apply \eqref{h:kappa} with $(h\,,f)$ replaced with $(|h|\,,|h|*|\tilde{h}|)$
in order to see that
\[
	\text{if}\quad\int_0^1\sup_{\|y\| > r}\big(|h|*|\tilde{h}|\big)(y)\omega_d(r)\,\d r<\infty,
	\quad\text{then}\quad |h|\in H_{-1}(\R^d).
\]
If $f\ge0$ and $x\mapsto f(x)$ is a radial function on $\R^d$ that decreases as $\|x\|$ increases,
then $\sup_{\|y\| > \|x\|}\big(|h|*|\tilde{h}|\big)(y)= f(x)$, and the above sufficient condition for $|h|=h$
to be in $H_{-1}(\R^d)$  appears
earlier in the literature, in the context of well-posedness for SPDEs.
See Dalang and Frangos \cite{DalangFrangos}, Karczewska
and Zabczyk\cite{KZ}, Peszat \cite{Peszat2002},
and Peszat and Zabczyk \cite{PeszatZabczyk2000}.
Closely-related results can be found in Cardon-Weber and Millet
\cite{CardonWeberMillet}, Dalang \cite{Dalang1999},
Foondun and Khoshnevisan \cite{FoondunKhoshnevisan2013},
and Millet and Sanz-Sol\'e \cite{MilletSanz}.

Recall the vector space $\cup_{p>1}\mathcal{G}_p(\R^d)$ (Definition
\ref{def:F_p}) and the inequalities of Proposition \ref{pr:PD}  in order to deduce the following.

\begin{lemma}\label{lem:Dalang:Lp}
	If $h\in\mathcal{G}_p(\R^d)$ for some $p>1$, then $|h|\in H_{-1}(\R^d)$.
	In particular, $h\in\cup_{p>1}\mathcal{G}_p(\R^d)$ implies that
	$\int_{\R^d}\bm{v}_\lambda(x)\big(|h|*|\tilde{h}|\big)(x)\,\d x<\infty$
	for some, hence all, $\lambda>0$.
\end{lemma}

In light of Theorem 1.2 of Foondun and Khoshnevisan \cite{FoondunKhoshnevisan2013},
Lemma \ref{lem:Dalang:Lp} implies a precise version of the somewhat
subtle assertion that sufficient  integrability of $h$ ensures
good decay at infinity of the Fourier transform of $|h|$.

\subsection{Proof of part 2 of Theorem \ref{th:1}}\label{sec:harmonic_3.3}

In this section, we prove part 2 of Theorem \ref{th:1}.
In fact, the work involves proving the following harmonic-analytic result.

\begin{proposition}\label{pr:HA}
	Suppose $f\in\mathfrak{M}_+(\R^d)$; that is, $f$ is a
	nonnegative-definite tempered Borel measure on $\R^d$.
	Suppose, in addition, that
	\begin{equation}\label{pre:Dalang}
		\int_{\R^d} \prod_{j=1}^d \left(\frac{1}{1\vee y_j^2}\right)\hat{f}(\d y)<\infty.
	\end{equation}
	Then, $\hat{f}\{0\}=0$ iff $f(\mathbb{B}_r)=o(r^d)$ as $r\to\infty$.
\end{proposition}

Proposition \ref{pr:HA} readily implies part 2 of Theorem \ref{th:1}
since  \eqref{Dalang} implies \eqref{pre:Dalang}. Therefore, it remains to prove
Proposition \ref{pr:HA}.

\begin{proof}[Proof of Proposition \ref{pr:HA}]
	Consider, for every real number $N>0$, the probability density function
	\begin{align}\label{I_N}
		I_N := N^{-d} \bm{1}_{[0,N]^d}\qquad\text{on $\R^d$}.
	\end{align}
	Then,
	\begin{equation}\label{I*I}
		(I_N*\tilde{I}_N)(x)  = N^{-d}\prod_{j=1}^d \left(1- \frac{|x_j|}{N}\right)_+
		\qquad\text{for every $x = (x_1\,,\ldots,x_d)\in\R^d$},
	\end{equation}
	where $a_+:=\max\{a\,,0\}$. Because
	$\tfrac12\bm{1}_{[-1/2,1/2]}(a)
	\le (1-|a|)_+ \le \bm{1}_{[-1,1]}(a)$
	for every $a\in\R$,
	\begin{equation}\label{III}
		(2N)^{-d}\bm{1}_{[-N/2, N/2]^d}
		\le I_N*\tilde{I}_N \le
		N^{-d}\bm{1}_{[-N, N]^d}\qquad\text{on $\R^d$}.
	\end{equation}
	
	Since $I_N*\tilde{I}_N \in C_c(\R^d)$ and the measure $f$ is locally finite,
	it follows that $I_N*\tilde{I}_N*f$ is continuous
	and bounded. Because $I_N*\tilde{I}_N*f$ is
	also nonnegative definite, it is therefore maximized at $0$. These properties,
	and the well-known fact that $\{\bm{p}_t\}_{t>0}$ is a (convolution)
	Feller semigroup, together
	imply that
	\[
		\left( I_N*\tilde{I}_N*f\right)(0) =
		\lim_{\varepsilon \rightarrow 0} \left( I_N*\tilde{I}_N*\bm{p}_{\varepsilon}*f
		\right)(0),
	\]
	where $\bm{p}_\varepsilon$ denotes the Gaussian heat kernel of \eqref{p}.
	We may apply Parseval's formula \eqref{Fourier_transform} next in order to see that
	\[
		\left(\bm{p}_{\varepsilon}*f\right) (x) =
		\int_{\R^d}\exp\left( ix\cdot y -\frac{\varepsilon}{2}\|y\|^2\right)
		\hat{f}(\d y)\qquad\text{for every $x\in\R^d$ and $\varepsilon>0$}.
	\]
	Therefore, Tonelli's theorem readily yield the identity,
	\[
		\left( I_N*\tilde{I}_N*\bm{p}_{\varepsilon}*f\right)(0)
		= 2^d \int_{\R^d}\hat{f}(\d y) \ \e^{-\varepsilon\|y\|^2/2}
		\prod_{j = 1}^d\frac{1 - \cos(Ny_j)}{(Ny_j)^2},
	\]
	where $2[1-\cos 0]/0^2:=1$.
	Let $\varepsilon\downarrow0$ and appeal to the monotone convergence theorem
	in order to arrive at the identity,
	\[
		\left( I_N*\tilde{I}_N*f \right)(0)
		= 2^d \int_{\R^d}\hat{f}(\d y) \
		\prod_{j = 1}^d\frac{1 - \cos(Ny_j)}{(Ny_j)^2}.
	\]
	Because $\prod_{j=1}^d [1 -\cos(a_j)]/a^2_j \le 2^{-d}
	\prod_{j=1}^d \min( 1\,, a_j^{-2})$
	for all $a=(a_1\,,\ldots,a_d)\in\R^d\setminus\{0\}$,
	the dominated convergence theorem and \eqref{pre:Dalang} together ensure that
	$( I_N*\tilde{I}_N*f)(0)$ converges to $\hat{f}(\{0\})$ as $N\to\infty$. Thus, we
	may deduce from \eqref{III} that
	\[
		2^{-d}\limsup_{N\to\infty}\frac{f\left([-N\,,N]^d\right)}{(2N)^d}  \le
		\hat{f}\{0\} \le 2^d\liminf_{N\to\infty} \frac{f\left([-N\,,N]^d\right)}{(2N)^d}.
	\]
	%[For the left-hand side, in this form, replace $N/2$ by $N$ in \eqref{III}.]
	Because $|\mathbb{B}_N|\propto N^d$ and
	$\mathbb{B}_N\subset[-N\,,N]^d\subset\mathbb{B}_{N\sqrt{d}}$, the above
	inequalities imply that $\hat{f}\{0\}=0$
	if and only if $f(\mathbb{B}_N)=o(|\mathbb{B}_N|)$ as $N\to\infty$.	
\end{proof}

\subsection{Proof of part 4 of Theorem \ref{th:1}}\label{subsec:part4}
Since $\hat{f}\in\mathfrak{M}_+(\R^d)$,
one can see easily that $I_N*\tilde{I}_N*\hat{f}$ is a continuous, nonnegative-definite
function for every $N>0$, where $I_N$ was defined in \eqref{I_N}. In particular,
\[
	(I_N*\tilde{I}_N*\hat{f})(x)\le(I_N*\tilde{I}_N*\hat{f})(0)
	\qquad\text{for every $x\in\R^d$ and $N>0$.}
\]
Multiply both sides by $N^d$ and let $N\to 0$ in order to deduce from \eqref{I*I} and
the dominated convergence theorem,
$\hat{f}\{x\} \le \hat{f}\{0\}$ for every $x\in\R^d$.
This completes the proof.\qed

\section{Proof of part 3 of Theorem \ref{th:1}}\label{sec:part3}
In the previous section we verified part 2 of Theorem \ref{th:1}.
Now we establish the third part of that theorem. Part 1 will be proved
a few sections hence.

Suppose there exists a number $c_0\in\R\setminus\{0\}$ such that
$\sigma(x)=c_0$ for all $x\in\R$.
In this case, the solution to \eqref{SHE} can be written, in mild form, as
\begin{align}\label{addive_noise}
	u(t\,, x) = 1 + c_0\int_{(0,t)\times\R^d}\bm{p}_{t-s}(x-z)
	\,\eta(\d s\,\d z).
   \end{align}
We see from this that, among other things,
$u(t)$ is a stationary, mean-one Gaussian random field. Dalang's theory
\cite{Dalang1999} ensures that $x\mapsto u(t\,,x)$ in continuous in $L^2(\Omega)$
for every $t>0$. Therefore, $x\mapsto u(t\,,x)$ has a Lebesgue-measurable
version (which we continue to write as $x\mapsto u(t\,,x)$);
see Remark \ref{rem:doob}.

Because of \eqref{addive_noise},
\begin{align*}
	& \int_{[0, N]^d}\d x
		\int_{[0, N]^d}\d y\
		{\rm Cov}\left( u(t\,, x)\,, u(t\,, y) \right)\\
	&= c_0^2\int_0^t\d s\int_{[0, N]^d}\d x
		\int_{[0,N]^d}\d y\
		\left\langle \bm{p}_s(x-\bullet) \,, \left(\bm{p}_s(y-\bullet)\right)
		*f\right\rangle_{L^2(\R^d)}
		\qquad\text{[by \eqref{Cov:eta}]}\\
	&= c_0^2\int_0^t\d s\int_{[0, N]^d}\d x
		\int_{[0,N]^d}\d y\int_{\R^{d}}\hat{f}(\d z)\
		\e^{iz\cdot(x - y) -s\|z\|^2},
\end{align*}
thanks to Parseval's identity \eqref{Fourier_transform}. Rearrange
the integrals, using Fubini's theorem, and compute directly
in order to find that
\[
	\int_{[0, N]^d}\d x
	\int_{[0, N]^d}\d y\
	{\rm Cov}\left( u(t\,, x)\,, u(t\,, y) \right)
	= 2^dc_0^2\int_{\R^{d}}\hat{f}(\d z) \frac{1 - \e^{- t \|z\|^2}}{\|z\|^2}
	\prod_{j = 1}^d\frac{1 - \cos(Nz_j)}{z_j^2},
\]
where $2[1-\cos0]/0^2 :=1$.
Since $f$ satisfies Dalang's condition
\eqref{Dalang}, the dominated convergence theorem
implies that
\begin{equation}\label{D1}
	\lim_{N \rightarrow \infty} \frac{1}{N^{2d}}\int_{[0, N]^d}\d x
	\int_{[0, N]^d}\d y\
	{\rm Cov}\left( u(t\,, x)\,, u(t\,, y) \right)=
	c_0^2 \: t\hat{f}(\{0\}).
\end{equation}

Now suppose, in addition, that $u$ is spatially ergodic.
Because $\E[u(t\,,x)]=1$ (see \eqref{addive_noise}),
von-Neumann's mean ergodic theorem (see for example Peterson \cite{Peterson}, and especially
Chapters 8 and \S9.3 of Edgar and Sucheston
\cite[Chapter 8 and Chapter 9, \S9.3]{EdgarSucheston}) implies that
$N^{-d}\int_{[0, N]^d}u(t\,, x)\,\d x$
converges in $L^2(\Omega)$ to $1$ as $N\to\infty$. Equivalently, that
\begin{equation}\label{D2}
	\lim_{N \rightarrow \infty} \frac{1}{N^{2d}}\int_{[0, N]^d}\d x
	\int_{[0, N]^d}\d y\
	{\rm Cov}\left( u(t\,, x)\,, u(t\,, y) \right) = 0.
\end{equation}
Part 3 of Theorem \ref{th:1} follows from comparing \eqref{D1} and \eqref{D2}.\qed

\section{Well posedness}

By Dalang \cite{Dalang1999},  equation \eqref{SHE} is well-posed when the spatial correlation
$f$  satisfies condition \eqref{Dalang}. In this section, we only prove the well posedness of \eqref{SHE}
when $f = h * \tilde{h}$ with $h\in\cup_{p>1}\mathcal{G}_p(\R^d)$.

\subsection{Stochastic convolutions}\label{sec:SC}
If $\Phi=\{\Phi(t\,,x)\}_{t\ge0,x\in\R^d}$ is a space-time random field, then
for all real numbers $\beta>0$ and $k\ge1$, we may define
\begin{equation}\label{N}
	\mathcal{N}_{\beta,k}(\Phi) := \sup_{t\ge0}\sup_{x\in\R^d}\e^{-\beta t}\|\Phi(t\,,x)\|_k.
\end{equation}
It is clear that $\Phi\mapsto\mathcal{N}_{\beta,k}(\Phi)$ defines a norm for
every choice of $\beta>0$ and $k\ge1$.
These norms were first introduced in \cite{FoondunKhoshnevisan2009}; see also
\cite{ConusKh2012}.
Corresponding to every $\mathcal{N}_{\beta,k}$, define $\mathbb{W}_{\beta, k}$ to
be the collection of all predictable random fields
$\Phi$ such that $\mathcal{N}_{\beta,k}(\Phi)<\infty$. We may think of elements of
$\mathbb{W}_{\beta,2}$ as \emph{Walsh-integrable random fields with Lyapunov exponent
$\le\beta$}. It is easy to see that each
$(\mathbb{W}_{\beta,k}\,,\mathcal{N}_{\beta,k})$
is a Banach space.

Suppose that the underlying probability space $(\Omega\,,\mathcal{F},\P)$ is large enough
to carry a space-time white noise $\xi$ (if not then enlarge it in the usual way).
Using that noise, we may formally
define, for every fixed measurable function $h:\R^d\to\R$,  a new noise $\eta^{(h)}$ as follows:
\begin{equation}\label{eta^h}
	\eta^{( h )}(\d s\,\d x) := \int_{\R^d} h (x-y)\,\xi(\d s\,\d y)\,\d x.
\end{equation}
Somewhat more precisely, if $H$ is a predictable random field such that
\[
	\E\int_0^t\d s\int_{\R^d}\d y\
	\left|  \left( H(s)*\tilde{ h }\right)(y)\right|^2<
	\infty\qquad\text{for every $t>0$},
\]
then Walsh's theory of stochastic integration ensures that the Walsh stochastic integral
\[
	\int_{(0,t)\times\R^d} H(s\,,x)\,\eta^{( h )}(\d s\,\d x)
	:= \int_{(0,t)\times\R^d} \left( H(s)*\tilde{ h }\right)(y)\,\xi(\d s\,\d y)
\]
is well-defined for every $t\ge0$, and in fact defines a continuous, mean-zero,
$L^2(\Omega)$ martingale indexed by $t\ge0$. Moreover, the variance of this martingale
at time $t>0$ is
\begin{align}\label{SID}
	\E\left(\left| \int_{(0,t)\times\R^d} H(s\,,x)\,\eta^{( h )}(\d s\,\d x)\right|^2\right)
		&=\E\int_0^t\d s\int_{\R^d}\d y\
		\left|  \left( H(s)*\tilde{ h }\right)(y)\right|^2\\\notag
	&=\int_0^t\d s\int_{\R^d}\d y\int_{\R^d}\d z\
		\E\left[H(s\,,y)H(s\,,z)\right]f(y-z),
\end{align}
provided, for example, that the preceding integral is absolutely convergent.
(As it is the case, here and elsewhere in this section, $f$ is defined in terms of $h\in
\cup_{p>1}\mathcal{G}_p(\R^d)$ via \eqref{f=h*h}.)

It is easy to see from this that $\eta^{(h)}$ is a particular construction of the noise
$\eta$ of the Introduction (see also Conus et al \cite{CJKS2013}), but has the advantage that
it provides a coupling $h\mapsto \eta^{(h)}$ that works simultaneously for many different
choices of $h$, whence spatial correlation functions $f$.

The preceding stochastic integration (see \eqref{SID}) frequently
allows for the integration of a large family of
predictable random fields $H$. The following simple result
highlights a large subclass of such  random fields when $h\in\cup_{p>1}\mathcal{F}_p(\R^d)$.

\begin{lemma}\label{lem:convergent}
	Suppose $ h \in\mathcal{F}_p(\R^d)$ for some $p>1$, and $H$ is
	a predictable process for which there exists a real number $r>0$ such that
	\begin{equation}\label{cond:H}
		\sup_{s\in[0,T]}\sup_{y\in\R^d}\E\left(|H(s\,,y)|^2\right)<\infty
		\quad\text{and}\quad
		\E\left(|H(t\,,x)|^2\right)=0\quad
		\text{for every $t>0$ and $x\in \mathbb{B}_r^c$}.
	\end{equation}
	Then, the final integral in \eqref{SID}
	is absolutely convergent and hence \eqref{SID} is valid  for every $t>0$.
\end{lemma}

\begin{proof}
	Choose and fix an arbitrary $t>0$. In accord with our earlier
	remarks, and thanks to the Cauchy--Schwarz inequality,
	it suffices to prove that
	\[
		J := \int_0^t\d s\int_{\mathbb{B}_r}\d y\int_{\mathbb{B}_r}\d z\ \|H(s\,,y)\|_2\|H(s\,,z)\|_2
		\left| f(y-z)\right|
		<\infty\qquad\text{for every $t>0$}.
	\]
	But the triangle inequality readily yields
	\[
		J\le |\mathbb{B}_r|
		\left(\int_0^t\sup_{y\in\R^d}\|H(s\,,y)\|_2^2\,\d s\right) \left(
		\int_{\mathbb{B}_{2r}} \left(| h |*|\tilde{ h }|\right)(w)\d w\right),
	\]
	which is finite thanks to \eqref{cond:H} and Proposition \ref{pr:PD}; see in particular
	\eqref{eq:PD:int}.
\end{proof}

The second portion of \eqref{cond:H} involves a compact-support condition which can sometimes
be reduced to a decay-type condition. We exemplify that next for a specific
family of the form $H(s\,,y) = \bm{p}_{t-s}(x-y)Z(s\,,y)$, where $t>s$ and $x\in\R^d$
are fixed and $\bm{p}$ denotes the heat kernel
[see \eqref{p}].
With this choice, the following ``stochastic convolution'' is a well-defined random field
provided that it is indeed defined properly as a Walsh integral for every $t>0$ and $x\in\R^d$:
\begin{equation}\label{SC}
	\left( \bm{p}\circledast Z\eta^{( h )}\right)(t\,,x) := \int_{(0,t)\times\R^d}
	\bm{p}_{t-s}(x-y) Z(s\,,y)\,\eta^{( h )}(\d s\,\d y).
\end{equation}

For every $k\ge2$, let $z_k^k$ denote the optimal constant of the $L^k(\Omega)$-form of
the Burkholder--Davis--Gundy inequality \cite{Burkholder,BDG,BG};
that is, for every continuous
$L^2(\Omega)$-martingale $\{M_t\}_{t\ge0}$,
and all real numbers $k\ge2$ and $t\ge0$,
\[
	\E\left( |M_t|^k\right) \le z_k^k\E\left( \langle M\rangle_t^{k/2}\right).
\]
Then,
\begin{equation}\label{z_k}
	z_2=1\qquad\text{and }\qquad
	z_k\le 2\sqrt{k}\qquad\text{for every $k>2$}.
\end{equation}
The first assertion is the basis of It\^o's stochastic calculus, and the second is due
to Carlen and Kree \cite{CarlenKree1991}, who also proved that $\lim_{k\to\infty} (z_k/\sqrt k) = 2$.
The exact value of $z_k$ is computed in the celebrated paper of Davis \cite{Davis1976}.

The following provides a natural condition for the stochastic convolution to be a well-defined
random field, the stochastic integral being defined in the sense of Walsh \cite{Walsh}, and extends
Proposition 6.1
of Conus et al \cite{CJKS2013} to the case that $f$ is possibly signed.
It might help to recall that $\bm{v}_\beta$ denotes the $\beta$-potential
kernel [see \eqref{HRD}].

\begin{lemma}[A stochastic Young inequality]\label{lem:Young}
	Suppose that $Z\in\mathbb{W}_{\beta,k}$ for some $\beta>0$, $k\ge2$,
	and that $ h \in\mathcal{G}_p(\R^d)$ for some $p>1$. Then,
	the stochastic convolution in \eqref{SC} is a well-defined Walsh
	integral,
	\[
		\mathcal{N}_{\beta,k}\left(\bm{p}\circledast Z\eta^{( h )}\right)
		\le z_k \: \mathcal{N}_{\beta,k}(Z) \sqrt{\frac12\int_{\R^d} \bm{v}_\beta(x)
		|f(x)|\,\d x},
	\]
	and the integral under the square root is finite.
\end{lemma}

\begin{proof}
	The integral under the square root is finite thanks to Lemma \ref{lem:Dalang:Lp}.
	We proceed to prove the remainder of the lemma.
	
	According to the theory of Walsh \cite{Walsh}, the random field
	$\bm{p}\circledast Z\eta^{( h )}$ is
	well defined whenever $\mathcal{Q}_2(t\,,x)<\infty$ where
	\[
		\mathcal{Q}_\kappa(t\,,x):=
		\int_0^t\d s\int_{\R^d}\d y\int_{\R^d}\d z\ \bm{p}_{t-s}(x-y)
		\bm{p}_{t-s}(x-z)
		\|Z(s\,,y)\|_k\|Z(s\,,z)\|_k| f(y-z) |
	\]
	for every $t>0$ and $x\in\R^d$.
	Moreover (see also \eqref{SID}), in that case, the Burkholder--Davis--Gundy
	inequality yields
	\begin{align*}
		&\E\left(\left|  \left(
			\bm{p}\circledast Z\eta^{( h )}\right)(t\,,x)\right|^k\right)\\
		&\le z_k^k\E\left(\left|
			\int_0^t\d s\int_{\R^d}\d y
			\int_{\R^d}\d z\ \bm{p}_{t-s}(x-y)\bm{p}_{t-s}(x-z)
			Z(s\,,y)Z(s\,,z) f(y-z)\right|^{k/2}\right)\\
		&\le z_k^k\left[ \int_0^t\d s\int_{\R^d}\d y
			\int_{\R^d}\d z\ \bm{p}_{t-s}(x-y)\bm{p}_{t-s}(x-z)
			\|Z(s\,,y)Z(s\,,z)\|_{k/2} |f(y-z)|\right]^{k/2}\\
		&\le z_k^k\left[ \mathcal{Q}_\kappa(t\,,x) \right]^{k/2},
	\end{align*}
	the last line holding thanks to the Cauchy--Schwarz inequality.
	It remains to prove that $\mathcal{Q}_\kappa(t\,,x)<\infty$ for all $t>0$ and $x\in\R^d$.
	
	Since $\|Z(s\,,y)\|_k\le \exp(\beta s)\mathcal{N}_{\beta,k}(Z)$ for all $s\ge0$ and $y\in\R^d$,
	it then follows that
	\begin{align*}
		\mathcal{Q}_\kappa(t\,,x)&\le \left[\mathcal{N}_{\beta,k}(Z)\right]^2
			\int_0^t\e^{-2\beta s}\,\d s\int_{\R^d}\d y\int_{\R^d}\d z\
			\bm{p}_{t-s}(x-y)\bm{p}_{t-s}(x-z)
			\left| f(y-z)\right|\\
		&\le \e^{2\beta t} \left[\mathcal{N}_{\beta,k}(Z)\right]^2
			\int_0^t\e^{-2\beta r}\,\d r \int_{\R^d}\d w\ \bm{p}_{2r}(w)
			|f(w)|,
	\end{align*}
	after two change of variables $[w=y-z,\ r=t-s]$, and
	thanks to the Chapman-Kolmogorov (semigroup)
	property of the heat kernel $\bm{p}$. Since
	\[
		\int_0^t \exp(-2\beta r)\bm{p}_{2r}(w)\,\d r\le\int_0^\infty
		\exp(-2\beta r)\bm{p}_{2r}(w)\,\d r=\tfrac12 \bm{v}_\beta(w),
	\]
	for every $w\in\R^d$ and $\beta>0$, this proves that
	\[
		\e^{-2\beta t}\mathcal{Q}_\kappa(t\,,x)
		\le \tfrac12\left[\mathcal{N}_{\beta,k}(Z)\right]^2
		\int_{\R^d} \bm{v}_\beta(w)
		|f(w)|.
	\]
	This inequality completes the proof of the lemma upon taking square roots,
	as the right-hand side of the preceding inequality is independent of $(t\,,x)$.
\end{proof}

\subsection{Well posedness}\label{sec:wellposedness}

Before we study the spatial ergodicity of the solution to \eqref{SHE} we address matters
of well posedness. As was mentioned earlier, well-posedness follows from the more general
theory of Dalang \cite{Dalang1999} when $h\ge0$, for example.
Here we say a few things about general
well posedness when $h$ is signed. This undertaking does require some new ideas, but most
of those new ideas have already been developed in the earlier sections, particularly as regards
the space $\cup_{p>1}\mathcal{G}_p(\R^d)$, which now plays a prominent role.

Recall the $\lambda$-potential $\bm{v}_\lambda$ from \eqref{HRD}.
Choose and fix a function $h\in\cup_{p>1}\mathcal{G}_p(\R^d)$
and recall from Lemma \ref{lem:Dalang:Lp}
that
\[
	\int_{\R^d}\bm{v}_\lambda(x)|f(x)|\,\d x\le
	\int_{\R^d}\bm{v}_\lambda(x)\left( |h|*|\tilde{h}|\right)(x)\,\d x<\infty,
\]
for one, hence all,
$\lambda>0$. As a consequence, we find that the following is a well-defined,
$(0\,,\infty)$-valued function on $(0\,,\infty)$:
\begin{equation}\label{Lambda}
	\Lambda_h(\delta) := \inf\left\{ \lambda>0:\ \int_{\R^d} \bm{v}_\lambda(x)
	\left( |h|*|\tilde{h}|\right)(x)\,\d x<
	\delta\right\}\qquad\text{for all $\delta>0$},
\end{equation}
where $\inf\varnothing:=\infty$.

\begin{theorem}\label{th:exist}
	Assume that $f = h* \tilde{h}$ with $h\in\mathcal{G}_p(\R^d)$ for some $p>1$.
	Then, the SPDE \eqref{SHE}, subject to non-random initial data $u(0)=u_0\in L^\infty(\R^d)$
	and non degeneracy condition $\text{\rm Lip}(\sigma)>0$,
	has a mild solution $u$ which is unique (upto a modification)
	subject to the additional condition \eqref{cond:moment}
	Finally, $(0,\,\infty)\times\R^d\ni(t\,,x)\mapsto u(t\,,x)$
	is continuous in $L^k(\Omega)$ for very $k\ge2$,
	and hence Lebesgue measurable (upto evanescence).
\end{theorem}

\begin{proof}[Outline of the proof of Theorem \ref{th:exist}]

	The proof follows a standard route. We therefore outline it,
	in part to document the veracity of the argument, but
	mainly as a means of introducing objects that we will need later on.
	
	Let $u_0(t\,,x):=u_0(x)$ for all $t\ge0$ and $x\in\R^d$, and define
	iteratively
	\begin{align}
		u_{n+1}(t\,,x) :=&\ \int_{\R^d}\bm{p}_t(y-x)u_0(y)\,\d y +
			\int_{(0,t)\times\R^d}\bm{p}_{t-s}(x-y)\sigma(u_n(s\,,y))  \notag
			\,\eta^{( h )}(\d s\,\d y)\\
		=&\ (\bm{p}_t*u_0)(t) + \left( \bm{p}\circledast  \label{dn1}
			\sigma(u_n)\eta^{( h )}\right)(t\,,x),
	\end{align}
	for every integer $n\ge0$ and all real numbers $t\ge0$ and $x\in\R^d$.
	Then, $u_0,u_1,\ldots$ represent the successive approximations of $u$ via Picard iteration.
	Since the first term is bounded uniformly by $\|u_0\|_{L^\infty(\R^d)}$,
	and since every $\mathcal{N}_{\beta,k}$ is a norm for every $\beta>0$
	and $k\ge1$,
	it follows that for all integers $n\ge0$, and all reals $\beta>0$ and $k\ge2$,
	\begin{equation}\label{N(u)}\begin{split}
		\mathcal{N}_{\beta,k}(u_{n+1}) &\le \|u_0\|_{L^\infty(\R^d)} +
			\mathcal{N}_{\beta,k}\left( \bm{p}\circledast
			\sigma(u_n)\eta^{( h )}\right)\\
		&\le\|u_0\|_{L^\infty(\R^d)} + z_k
			\mathcal{N}_{\beta,k}\left( \sigma(u_n)\right)\sqrt{\frac12\int_{\R^d}
			\bm{v}_\beta(x) |f(x)|\,\d x};
	\end{split}\end{equation}
	see Lemma \ref{lem:Young}.
	Because $|\sigma(z)|\le|\sigma(0)|+\text{\rm Lip}(\sigma)|z|$
	for all $z\in\R$, it follows that
	\[
		\mathcal{N}_{\beta,k}(u_{n+1}) \le\|u_0\|_{L^\infty(\R^d)} + z_k
		\left(|\sigma(0)|+\text{\rm Lip}(\sigma)\mathcal{N}_{\beta,k}
		( u_n)\right)
		\sqrt{\frac12\int_{\R^d}
		\bm{v}_\beta(x) \left( |h|*|\tilde{h}|\right)(x)\,\d x}.
	\]
	This is valid for every $\beta>0$ and $k\ge2$.

    Choose and fix $\varepsilon \in (0, 1)$.
    Because
	\begin{equation}\label{E:Lambda-Beta}
		 \beta\ge \Lambda_h
		 \left(\frac{2(1-\varepsilon)^2}{[z_k\text{\rm Lip}(\sigma)]^2}\right)
		 \quad\text{iff}\quad
		 \int_{\R^d}\bm{v}_\beta(x) \left( |h|*|\tilde{h}|\right)(x)\,\d x
		 \le \frac{2(1-\varepsilon)^2}{[z_k\text{\rm Lip}(\sigma)]^2},
	\end{equation}
	it follows that,
	under the condition $\beta \ge \Lambda_h(2(1-\varepsilon)^2/[z_k\text{\rm Lip}(\sigma)]^2)$,
	\begin{equation}\label{eq:N(u_n+1)}\begin{split}
		\mathcal{N}_{\beta,k}(u_{n+1}) &\le\|u_0\|_{L^\infty(\R^d)} + z_k
			|\sigma(0)|\sqrt{\frac12\int_{\R^d}
			\bm{v}_\beta(x) \left( |h|*|\tilde{h}|\right)\,\d x} + (1-\varepsilon)
			\mathcal{N}_{\beta,k}(u_n)\\
		&\le \|u_0\|_{L^\infty(\R^d)} +
			\frac{|\sigma(0)|}{\text{\rm Lip}(\sigma)} +
			(1-\varepsilon)\mathcal{N}_{\beta,k}(u_n)\\
		&\le \|u_0\|_{L^\infty(\R^d)} +
			\frac{|\sigma(0)|}{\text{\rm Lip}(\sigma)} + (1-\varepsilon)\left[
			\|u_0\|_{L^\infty(\R^d)} +
			\frac{|\sigma(0)|}{\text{\rm Lip}(\sigma)}\right] + (1-\varepsilon)^2
			\mathcal{N}_{\beta,k}(u_{n-1})\\
		&\le\cdots\le \left[\|u_0\|_{L^\infty(\R^d)} +
			\frac{|\sigma(0)|}{\text{\rm Lip}(\sigma)}\right]\cdot
			\left[\sum_{j=0}^n(1-\varepsilon)^j +
			(1-\varepsilon)^{n+1}\|u_0\|_{L^\infty(\R^d)}\right]\\
		&\le \left[ \|u_0\|_{L^\infty(\R^d)} +
			\frac{|\sigma(0)|}{\text{\rm Lip}(\sigma)}\right]\cdot
			\left[\frac1\varepsilon +
			(1-\varepsilon)^{n+1}\|u_0\|_{L^\infty(\R^d)}
			\right],
	\end{split}\end{equation}
	after iteration. Similarly, one finds that
	\begin{equation}\label{u_n-u}\begin{split}
		\mathcal{N}_{\beta,k}(u_{n+1} - u_n) &\le
			\mathcal{N}_{\beta,k}\left( \bm{p}\circledast
			\left[ \sigma(u_n)-\sigma(u_{n-1})\right]\eta^{( h )}\right)\\
		&\le z_k\mathcal{N}_{\beta,k}\left( \sigma(u_n)-\sigma(u_{n-1})\right)
			\sqrt{\frac12\int_{\R^d}
			\bm{v}_\beta(x) \left( |h|*|\tilde{h}|\right)\,\d x}\\
		&\le z_k\text{\rm Lip}(\sigma)\mathcal{N}_{\beta,k}
			\left( u_n-u_{n-1}\right)\sqrt{\frac12\int_{\R^d}
			\bm{v}_\beta(x) \left( |h|*|\tilde{h}|\right)\,\d x}\\
		&\le(1-\varepsilon)\mathcal{N}_{\beta,k}\left( u_n-u_{n-1}\right),
	\end{split}\end{equation}
	provided still that
	$\beta \ge \Lambda_h(2(1-\varepsilon)^2/[z_k\text{\rm Lip}(\sigma)]^2)$.
	It follows immediately that $\{u_n\}_{n\ge0}$ is a Cauchy sequence in $\mathbb{W}_{\beta, k}$
	when $\beta\ge\Lambda_h(2(1-\varepsilon)^2/[z_k\text{\rm Lip}(\sigma)]^2)$. It also implies readily
	that  $u:=\lim_{n\to\infty}u_n$ is an element of $\mathbb{W}_{\beta, k}$, for the same range
	of $\beta$'s, and that $u$ solves \eqref{SHE}. This and Fatou's lemma together prove the asserted
	upper bound for $\E(|u(t\,,x)|^k)$ as well.
	
	The proof of uniqueness is also essentially standard: Suppose there existed
	$u,v\in\mathbb{W}_{\beta, k}$ for some
	$\beta \ge \Lambda_h(2(1-\varepsilon)^2/[z_k\text{\rm Lip}(\sigma)]^2)$
	both of which are mild solutions to \eqref{SHE}. Then, the same argument that led to
	\eqref{u_n-u} yields
	\[
		\mathcal{N}_{\beta,k,T}(u-v)\le(1-\varepsilon)\mathcal{N}_{\beta,k,T}(u-v),
	\]
	for all $\beta \ge \Lambda_h(2(1-\varepsilon)^2/[z_k\text{\rm Lip}(\sigma)]^2)$
	and $T>0$, where
	\[
		\mathcal{N}_{\beta,k,T}(\Phi) := \sup_{t\in[0,T]}\sup_{x\in\R^d}
		\e^{-\beta t}\|\Phi(t\,,x)\|_k;
	\]
	compare with \eqref{N}.
	In particular, it follows that there exists $\beta>0$ such that
	\[
		\mathcal{N}_{\beta,k,T}(u-v)=0 \qquad\text{for all $T>0$},
	\]
	and hence $u$ and $v$ are modifications of one another. We can unscramble the latter
	displayed statement in order to see that this yields the asserted bound for
	$\E(|u(t\,,x)|^k)$. Similarly, one proves $L^k(\Omega)$ continuity, which completes our (somewhat
	abbreviated) proof of Theorem \ref{th:exist}.
\end{proof}

\begin{remark}\label{Picard_iteration}
	Let us pause and record the following --- see \cite[eq.\ (54)]{Dalang1999} --
	ready by-product of Theorem \ref{th:exist} and the Lipschitz continuity of $\sigma$:
	For all $T >0$ and $k \geq 2$,
	\begin{align}\label{C_{T, k}}
		C_{T, k}:=\sup_{n\ge0}\sup_{(t\,,x)\in[0, T]\times\R^d}\E\left(|
		\sigma(u_n(t\,,x))|^k\right) < \infty,
	\end{align}
	where $u_n$ denotes the $n$th-stage Picard iteration of the SPDE \eqref{SHE}. Eq.\
	\eqref{C_{T, k}} is  valid also
	in the case that $f$ satisfies \eqref{f=h*h}; see \eqref{eq:N(u_n+1)} for some
	$h\in\cup_{p>1}\mathcal{G}_p(\R^d)$.
\end{remark}

\begin{remark}\label{rem:doob}
	Because of $L^k(\Omega)$-continuity, Doob's theory of separability becomes applicable (see
	Doob \cite{Doob})
	and implies, among other things, that $x\mapsto u(t\,,x)$ is Lebesgue measurable.
	This is of course directly relevant to the present discussion of spatial ergodicity.
\end{remark}

\section{Malliavin calculus}\label{sec:Malliavin}
\subsection{A Clark--Ocone formula and a Poincar\'e inequality}

Suppose that  the spatial correlation $f$ of the noise is a measure
that satisfies Dalang's condition
\eqref{Dalang}, or  is a function of the form $f= h* \tilde{h}$ where $|h| \in H_{-1} (\R^d)$.
Let $\HH_0$ be the Hilbert space defined as the completion of
$C_c^\infty(\R^d)$ under the scalar product
\[
	\langle \phi\,, \varphi \rangle_{\HH_0}
	= \langle \phi\,, \varphi*f\rangle_{L^2(\R^d)},
\]
and let $\HH := L^2(\R_+\,; \HH_0)$. Then, the Gaussian family $\{\eta(\phi)\}_{\phi\in \HH}$,
described by the family of Walsh-type stochastic integrals,
\[
	\eta(\phi)= \int_{\R_+\times\R^d}  \phi(s\,,x)\, \eta(\d s\, \d x),
\]
defines an isonormal Gaussian process on the Hilbert space $\HH$.
When  $f= h* \tilde{h}$, we can
use the noise $\eta^{(h)}$ to construct this integral from the integral
with respect to a space-time white noise as it has been done in  \S\ref{sec:SC}.

  In this framework, we can develop the Malliavin calculus with respect to the noise $\eta$. Next we recall some of the basic definitions of that Malliavin calculus.

Denote by $\mathcal {S}$ be the set of smooth and cylindrical random variables of the form
\[
	F= \Psi(\eta(\phi_1)\,, \dots, \eta(\phi_n)),
\]
where $\Psi\in C_c^\infty(\mathbb{R}^n)$ and
$\phi=(\phi_1\,,\ldots,\phi_n) \in \HH^n$.
If $F\in \mathcal{S}$ has the above form, then the
\emph{Malliavin derivative} $DF$ is the $\HH$-valued random variable defined by
\[
 	DF := (\nabla\Psi)\left( \eta(\phi_1)\,,\ldots,\eta(\phi_n)\right)\cdot
	\phi = \sum_{i=1}^n \left( \partial_i \Psi \right)
	\left(\eta(\phi_1) \dots, \eta(\phi_n) \right) \phi_i.
\]
In particular, $D(\eta(\varphi))= \varphi$ or every $\varphi\in\HH$; that is, $D$ can be interpreted as
the inverse of the Wiener stochastic-integral operator $\phi\mapsto \eta(\phi)$.

The operator $D$ is a closable linear mapping
from $L^p(\Omega)$ to $L^p(\Omega\,; \HH)$ for every real number $p\ge 1$.
We can define the Gaussian Sobolev space $\mathbb{D}^{1,p} $
as the closure of $\mathcal{S}$ with respect to the seminorm $\|\cdots\|_{1,p}$,
defined via
\[
	\|F\|_{1,p}^p := \E\left(|F|^p\right) +
	\E\left(  \|DF\|_{\HH}^p  \right).
\]
We will make use of the notation $D_{s,z}F$ to represent  the derivative as a random field,
indexed by $(s\,,z)\in\R_+\times\R^d$. In particular, if $F=u(t\,,x)$, then
$D_{r,z}u(t\,,x)$ will serve as short-hand for $D_{r,z}[ u(t\,,x)]$.

The \emph{divergence operator} $\delta $ is defined as the adjoint of $D$.
More precisely,  we first define the \emph{domain of  $\delta$}  in $L^2(\Omega)$ ---
denoted by ${\rm Dom}\, \delta$ ---  as the set of random
elements $v\in  L^2(\Omega\,  ;  \HH  ) $ for which we can find a real number $c_v>0$
such that
\[
	\left| \E \left( \langle v\,,DF \rangle_{\HH} \right) \right| \le c_v \| F\|_{L^2(\Omega)}
	\qquad\text{for every $F\in\mathbb{D}^{1,2}$.}
\]

For every $v\in{\rm Dom}\, \delta$, we define the real-valued
random variable $\delta(v)$ via the following \emph{duality relation}:
\begin{equation}\label{duality}
	\E\left[\langle DF\,, v\rangle_{\HH}  \right] = \E[F\delta(v)]
	\qquad\text{for every $F\in\mathbb{D}^{1,2}.$}
\end{equation}

It turns out that $\delta$ is a closed operator.
This means that, if $v_1,v_2,\ldots \in {\rm Dom}\, \delta$ satisfy $\lim_{n\to\infty}v_n=v$
in $L^2(\Omega \, ;  \HH  )$, and if
$G:=\lim_{n\to\infty}\delta(v_n)$ exists in $L^2(\Omega)$, then
$v\in {\rm Dom}\, \delta$, and $\delta(v)=G$.

Next, we provide 2 examples of elements of the domain of the divergence operator $\delta$.

\begin{example}
	Suppose that $v\in L^2(\Omega \times \R_+\,; \HH_0)$ is a
	smooth and cylindrical  $\HH_0$-valued stochastic process of the form
	$v(t)=\sum_{j=1}^{n}F_{j} \phi _{j}(t)$
	where the $F_{j}\in \mathcal{S}$, and $\phi_{j}\in \HH_0$
	for all $j=1,\ldots,n$.
	Then, $v\in{\rm Dom}\,\delta$, and
	\[
		\delta (v)=\sum_{j=1}^{n}F_{j}\eta(\phi_{j})-\sum_{j=1}^{n}\left\langle
		DF_{j}\,,\phi_{j}\right\rangle _{\HH}.
	\]
	This property follows immediately from \eqref{duality}.
\end{example}

\begin{example}
	Consider a predictable random field  $\{H(s\,,y)\}_{s\ge 0, y\in \R^d}$
	such that the Walsh integral $\int_{\R_+\times\R^d}H\,\d\eta$ is well defined
	and in $L^2(\Omega)$.
	  Then, $H\in{\rm Dom}\,\delta$ as an $\HH_0$-valued process,
	  and $\delta(H)$ coincides with the Walsh stochastic integral of $H$; that is,
	  $\delta(H)= \int_{\R_+\times \R^d} H\,\d\eta.$
	  This result is well-known in the case of stochastic integrals with
	  respect to the Brownian motion (see \cite{GT82} and also \cite[Proposition 1.3.11]{Nualart}).
	  The same proof works for $\HH_0$-valued processes.
\end{example}

	The Clark--Ocone formula will play a fundamental role in the proof of
	our results. We state below this formula and give a proof for the sake of completeness.
	Throughout, we denoted by $\mathcal{F}:=\{\mathcal{F}_t\}_{t\ge0}$
	the natural filtration of the noise $\eta$;
	that is, $\mathcal{F}$ is the usual augmentation of the filtration $\mathcal{F}^0$, defined via
	$\mathcal{F}^0_0:=\{\varnothing\,,\Omega\}$ and
	\[
		\mathcal{F}_t^0 := \text{sigma-algebra generated by }
		\left\{ \int_{(0,r)\times\R^d}\phi(x)\,\eta(\d s\,\d x);\, 0\le r\le t\right\}
		\quad\text{for all $t>0$},
	\]
	as $\phi$ ranges over all elements of $C_c(\R^d)$; see also \eqref{Cov:eta}.

\begin{proposition}[A Clark--Ocone formula/Poincar\'e inequality]\label{Clark-Ocone}
	For every $F\in  \mathbb{D}^{1,2}$,
	\[
		F= \E F  + \int_{\R_+\times\R^d}
		\E\left(D_{s,z} F \mid \mathcal{F}_s\right) \eta(\d s \, \d z)
		\qquad\text{a.s.}
	\]
	Consequently, we have the Poincar\'e inequality,
	$\Var(F) \le \E( \|DF\|_{\HH}^2).$
\end{proposition}

\begin{proof}
	It suffices to prove the integral representation of $F$; the
	Poincar\'e inequality follows from the integral representation
	by the same argument as in \eqref{CHL}, using the spatial
	covariance structure of $\eta$; see \eqref{Cov:eta}.
	
	One can extend  the martingale representation theorem,
	proved by Doob in \cite{Doob}, to martingales
	that take value in a Hilbert space \cite{LO73, Ou75}.
	It follows from that extension that there exists a unique
	$\HH_0$-valued predictable process $H$ such that the Walsh integral
	$\int_{\R_+\times\R^d}H\,\d\eta$ is well-defined in $L^2(\Omega)$,
	and
	\begin{equation}\label{Ito(F)}
		F= \E F  + \int_{\R_+\times\R^d} H(s\,,z) \,\eta(\d s \, \d z).
	\end{equation}
	It remains to prove that
	\begin{equation}\label{H=D}
		H(s\,,z)=\E( D_{s,z} F \mid \mathcal{F}_s),
	\end{equation}
	viewed as in identity in $L^2(\Omega \times \R_+; \HH_0)$.
	We will prove \eqref{H=D} in the case that $f$ is a nonnegative-definite  \emph{function}.
	The more general case where $f$ is a nonnegative distribution follows in exactly
	the same way, but one has to adjust the ensuing integrals for example so that $\int \psi(x)f(x)\,\d x$ is replaced
	by $f(\psi)$, etc. Proving the more general case requires no new ideas, only the introduction of
	heavy-handed notation. Therefore, we stick to the less notation-intensive
	case that $f$ is a function.
	
	Both sides of \eqref{H=D} define predictable random fields.
	Therefore, it suffices to prove that
	\begin{align*}
		&\int_0^\infty \d s \int_{\R^d} \d y \int_{\R^d} \d z \
			\E[ H(s\,,y) v(s\,,z)] f(y-z)\\
		&\hskip2in=
			\int_0^\infty \d s \int_{\R^d} \d y \int_{\R^d} \d z \
			\E[ \E(D_{s,y} F \mid \mathcal{F}_s)  v(s\,,z)] f(y-z),
	\end{align*}
	for every predictable process $v\in L^2(\Omega \times \R_+\,; \HH_0)$.
	Since $v(s\,,\cdot)$ is $\mathcal{F}_s$-measurable for every $s>0$, the duality relation \eqref{duality}
	implies that
	\begin{align*}
		&\int_0^\infty \d s \int_{\R^d} \d y \int_{\R^d} \d z \
			\E[ \E (D_{s,y} F \mid \mathcal{F}_s)  \cdot v(s\,,z)] f(y-z) \\
		&\hskip1in=  \int_0^\infty \d s \int_{\R^d} \d y \int_{\R^d} \d z \
			\E\left[ D_{s,y} F  \cdot v(s\,,z)\right] f(y-z)
			= \E  \left[ \langle DF, v \rangle_{\HH} \right] = \E[ F \delta(v)].
	\end{align*}
	Because $\delta(v) $ coincides with the Walsh integral of $v$,
	\eqref{Ito(F)} implies that
	\[
		\E[F\delta(v)] = \E\left[\int_{\R_+\times\R^d}H\,\d\eta\cdot
		\int_{\R_+\times\R^d}v\,\d\eta\right] =
		\int_0^\infty \d s \int_{\R^d} \d y \int_{\R^d} \d z \
		\E[ H(s\,,y) v(s\,,z)]f(y-z),
	\]
	thanks to the $L^2(\Omega)$-isometry  of the Walsh stochastic integral.
	This concludes the proof.
\end{proof}

\subsection{Differentiability of the solution}\label{sec:diff}
In order to apply Malliavin calculus, we first need to check that the solution
to \eqref{SHE} is differentiable in the sense of Malliavin calculus.
This section is concerned with that, which we state next in the following
comprehensive form.

\begin{theorem}\label{malliavin_derivative}
	Suppose $f$ satisfies either Dalang's condition \eqref{Dalang}, or
	condition \eqref{f=h*h} with $h\in\cup_{p>1}\mathcal{G}_p(\R^d)$,
	and let $u$ denote the mild solution to \eqref{SHE}. Then,
	\[
		u(t\,,x)\in\bigcap_{k\ge2}\mathbb{D}^{1,k}
		\qquad\text{for  every $(t\,,x) \in (0\,,\infty)\times \R^d$.}
	\]
	Moreover, if $t\in(0\,,T)$ for a fixed $T>0$ and $x\in\R^d$, then
	\begin{equation}\label{L^p_norm}
		\| D_{s,y} u(t\,,x) \|_{k} \le \frac{2C_{T,k}\e^{\lambda_0(t -s)}}{\sqrt{%
		1 - 2^{(d-2)/2} \left[z_k{\rm Lip} (\sigma)\right]^2
		(\bm{v}_{\lambda_0}*\bar{\bf f})(0)}}
		\,\bm{p}_{t-s}(x-y),
	\end{equation}
	for almost every $(s\,,y) \in (0\,,t) \times \R^d$.
	The quantities $C_{T,k}$, $z_k$, and $\bm{v}_{\lambda_0}$ are respectively
	defined in \eqref{C_{T, k}}, \eqref{z_k}, and \eqref{HRD},
	$\lambda_0>0$ is arbitrary but large enough to ensure that
	\[
		( \bm{v}_{\lambda_0}*\bar{\bf f} )(0) <
		\frac{1}{2^{(d-2)/2}\left[ z_k{\rm Lip} (\sigma)\right]^2},
	\]
	and
	\begin{align}\label{E:bar-f}
		\bar{\bf f} :=\begin{cases}
			f&\text{when $f$ satisfies \eqref{Dalang},}\\
			|h|*|\tilde{h}|&\text{when $f$ satisfies \eqref{f=h*h} for some
				$h\in\bigcup_{p>1}\mathcal{G}_p(\R^d)$}.
		\end{cases}
	\end{align}
\end{theorem}

\begin{remark}
	In the case that  $f$ is a nonnegative \emph{function} that
	satisfies Dalang's condition \eqref{Dalang},
	the first part of this theorem ---
	namely that $u(t\,,x)\in\cap_{k\ge2}\mathbb{D}^{1,k}$ ---
	was proved by Chen and Huang \cite[Proposition 3.2]{ChenHuang2019}.
\end{remark}

The proof of Theorem \ref{malliavin_derivative} first requires some preliminary
development, which we present as two lemmas.

\begin{lemma}\label{inequality:g}
	Choose and fix a real number $T>0$. Then, for
	every nondecreasing function $g: [0\,, T] \mapsto \R_+$
	and  for all $t \in (0\,, T)$ and $y \in \R^d$,
	\[
		\int_0^tg(s)\bm{p}_{2s(t - s)/t}(y)\,\d s \le
		2^{(d+2)/2}\int_{0}^{t}g(s)\bm{p}_{2(t - s)}(y)\,\d s.
	\]
\end{lemma}

\begin{proof}
	The proof is similar to Lemma B.1 of Chen and Huang
	\cite{CH19Comparison}. Since $g$ is monotone,
	\[
		\int_0^{t/2}g(s)\bm{p}_{2s(t - s)/t}(y)\,\d s
		= \int_{t/2}^{t}g(t  - s)\bm{p}_{2s(t - s)/t}(y)\,\d s
		\le  \int_{t/2}^{t}g(s)\bm{p}_{2s(t - s)/t}(y)\,\d s.
	\]
	Hence,
	\begin{align*}
		\int_0^t g(s)\bm{p}_{2s(t - s)/t}(y)\,\d s
			&\le  2\int_{t/2}^{t} g(s)\bm{p}_{2s(t - s)/t}(y)\,\d s\\
		&= 2\int_{t/2}^{t} g(s) \exp\left(-\dfrac{t\|y\|^2}{4s(t - s)}\right)
			\frac{\d s}{(4\pi s(t - s)/t)^{d/2}}\nonumber\\
		& \le  2^{(d+2)/2}\int_{t/2}^{t} g(s)\exp\left(-\dfrac{\|y\|^2}{4(t - s)}\right)
			\frac{\d s}{(4\pi (t - s))^{d/2}}\\
		&=2^{(d+2)/2}\int_{t/2}^{t} g(s) \bm{p}_{2(t - s)}(y)\,\d s ,
	\end{align*}
	which clearly is bounded from above by $2^{(d+2)/2}\int_0^t g(s)
	\bm{p}_{2(t-s)}\,\d s$.
\end{proof}

In order to estimate the $L^k$-norm of the Malliavin derivative of the solution, we introduce some notations which will be used later on.
\begin{equation}\label{function:k}
	\kappa(t):= \left( \bm{p}_{2t}*\bar{\bf f}\right)(0),
\end{equation}
for the same distribution $\bar{\bf f}$ that was defined in the statement of Theorem \ref{malliavin_derivative}.
Next, define $h_0(t) \equiv 1$ and
\begin{equation}\label{function:h_n}
	h_n(t):= \int_0^t h_{n - 1}(s)\kappa(t - s)\,\d s
	\qquad\text{for all $t>0$ and $n\ge1$}.
\end{equation}
By induction, it is clear that the function $h_n$ is nondecreasing for all $n \geq 0$.
[The functions $\{h_n\}_{n\ge0}$ should not be confused with the function $h$ in
\eqref{f=h*h}.]

We now follow Chen and Huang \cite{CH19Comparison}
and define for every $\gamma \geq 0$ and $t > 0$,
\begin{align}\label{function:H}
	H(t\,; \gamma):= \sum_{n = 0}^{\infty}\gamma^nh_n(t).
\end{align}
Recall \eqref{HRD}.

\begin{lemma}\label{H:finite}
	Suppose $f$ satisfies either \eqref{Dalang}, or
	\eqref{f=h*h} with $h\in\cup_{p>1}\mathcal{G}_p(\R^d)$.
	Then, for all $\gamma \geq 0$ and $t \geq 0$, the following inequality holds
	\[
		H(t\,; \gamma) \le  \frac{\e^{2\lambda t}}{1 - \tfrac12
		\gamma \left(\bm{v}_{\lambda}*\bar{\bf f}\right)(0)}
		\qquad\text{for all $t>0$},
	\]
	provided that $\lambda > 0$ and $0\le \gamma < 2/(\bm{v}_{\lambda}*\bar{\bf f})(0)$.
\end{lemma}

\begin{proof}
	Define $\mu_n := \sup_{t>0} [ \e^{-2\lambda t} h_n(t)]$ for every integer $n\ge0$, and
	note that
	\[
		h_{n + 1}(t) = \int_0^t h_{n}(s)\kappa(t - s)\,\d s
		\le \e^{2\lambda t}\mu_n\int_0^t \e^{-2\lambda (t-s)}\kappa(t-s)\,\d s
		\le \e^{2\lambda t}\mu_n\int_0^\infty \e^{-2\lambda s}\kappa(s)\,\d s,
	\]
	for all $n\ge0$ and $t>0$. Thus,
	$\mu_{n+1} \le \frac12\mu_n(\bm{v}_\lambda*\bar{\bf f})(0)$
	for all $n\ge0$. Since  $\mu_0=1$, we are led to
	\[
		h_n(t) \le \left[ \tfrac12\left(\bm{v}_\lambda*\bar{\bf f}\right)(0)\right]^n
		\qquad\text{for all $t>0$ and $n\ge0$},
	\]
	which leads to the lemma summarily.
\end{proof}

We are now ready to prove Theorem \ref{malliavin_derivative}.

\begin{proof}[Proof of Theorem \ref{malliavin_derivative}]
	Throughout, we choose and fix a real number $T>0$.
	We will prove the result in the case that $f$ is a \emph{function} that
	satisfies either \eqref{Dalang}, or \eqref{f=h*h} for some
	$h\in\cup_{p>1}\mathcal{G}_p(\R^d)$. The remaining case is when $f$
	is a measure that satisfies Dalang's condition \eqref{Dalang}; that case is proved
	by making only small adjustments to the following argument, but requires the introduction
	of a good deal of notation. Therefore, we consider only the case that $f$ is a function.
	Note in particular, that $\bar{\bf f}$ is also a function,
	and regardless of whether or not $f$ is signed, we always have
	$|f|\le |\bar{\bf f}|$.
	From here on, we adapt the iterative method of \cite[Lemma 2.2]{CH19Comparison}.
	
	Let $u_0(t\,,x) :=1$ for all $t>0$ and $x\in\R^d$, and
	recall  the Picard iterations introduced in \eqref{dn1}:
	\[
		u_{n+1}(t\,,x):= 1 + \int_{\R_+\times\R^d} \bm{p}_{t-s} (x-y) \sigma(u_n(s\,,y)) \,\eta(\d s\,\d y),
	\]
	for all $n\ge0$, $t>0$, and $x\in\R^d$.
	
		Let $C_{T, k}$ be the constant in \eqref{C_{T, k}} and define
	\[
		\gamma:= 2^{d/2} \left[ z_k  {\rm Lip} (\sigma)\right]^2,
	\]
	where $z_k$ was defined in \eqref{z_k}.
	We claim that, for the above choice of $\gamma$,
	$u_{n}(t\,,x)\in \mathbb{D}^{1,k}$
	for every $(t\,,x) \in (0\,,T)\times \R^d$ and $k\ge 2$,
	and
	\begin{equation} \label{e2}
		\| D_{s,y} u_{n}(t\,,x) \|_{k} \le \sqrt{2}C_{T,k}\,
		\bm{p}_{t-s}(x-y)\left(\sum_{i = 0}^{n}\gamma^ih_i(t - s)\right)^{1/2},
	\end{equation}
	for almost every $(s\,,y) \in (0\,,t) \times \R^d$. Let ($P_n$) denote this logical
	proposition. Clearly ($P_0$) holds, as the left-hand side of \eqref{e2} is equal to zero.
	Now suppose ($P_k$) holds for every integer $k=0\,,\ldots,n$, where $n\ge0$ is
	integral. We propose to derive the conditional truth of ($P_{n+1})$. This will be enough to prove \eqref{e2}
	inductively.
	
	According to Proposition 1.2.4 of Nualart \cite{Nualart},
	$\sigma(u_n(t\,,x))\in\mathbb{D}^{1,k}$  for every $(t\,,x) \in(0\,,T)\times \R^d$;
	moreover,
	\[
		D( \sigma( u_n(t\,,x))) = \Sigma_n Du_n(t\,,x)
		\qquad\text{a.s.}
	\]
	where $\Sigma_n:=\sigma'(u_n(t\,,x))$ for any version of the
	derivative $\sigma'$. This is  because, on the event $\{ \| Du_n(t\,,x)\|_{\HH} >0\}$,
	the random variable $u_n(t\,,x)$ is absolutely continuous.

	We apply the properties of the divergence operator (see \cite[Prop.\ 1.3.8]{Nualart})
	in order to find that
	$ \int_{(0,t)\times\R^d} \bm{p}_{t-s} (x-y) \sigma(u_n(s\,,y))\, \eta(\d s \,\d y)
	\in \mathbb{D}^{1,k} $. Moreover,
	\begin{align*}
		D_{r,z} u_{n+1}(t\,,x) &=
			D_{r,z} \left(\int_{(0,t)\times\R^d} \bm{p}_{t-s} (x-y) \sigma(u_n(s\,,y))\,
				\eta(\d s \,\d y) \right)\\
		 &= \bm{p}_{t-r} (x-z) \sigma(u_n(r\,,z)) + \int_{(r,t)\times\R^d} \bm{p}_{t-s} (x-y)
		 	\Sigma_n  D_{r,z} u_n(s\,,y )\,\eta(\d s \,\d y)
			\qquad\text{a.s.,}
	 \end{align*}
	 whence
	 \begin{align*}
		 & \left\|   D_{r,z} u_{n+1}(t\,,x)  \right\|_k \le
			\bm{p}_{t-r} (x-z)  \left\| \sigma(u_n(r\,,z))  \right\|_k \\
		& \hskip2.5in + \left\|  \int_{(r,t)\times\R^d}
			\bm{p}_{t-s} (x-y)  \Sigma_n  D_{r,z} u_n(s\,,y )\,\eta(\d s \,\d y)
			\right\|_k,
	\end{align*}
	for every integer $k\ge2$.
	  Define
	\[
		\mathcal{P}_\tau(y\,,w\,;x) := \bm{p}_\tau (x-y)\bm{p}_\tau (x-w)
		\qquad\text{for every $\tau>0$ and $x,y,w\in\R^d$}.
	\]
	Then, the Burkholder--Davis--Gundy inequality \cite{Burkholder,BDG,BG} implies that
	\begin{align*}
		& \E \left( \left|  \int_{(r,t)\times\R^d} \bm{p}_{t-s} (x-y)
			\Sigma_n  D_{r,z} u_n(s\,,y )\,\eta(\d s \,\d y)\right |^k \right)  \\
		& \le \left[z_k{\rm Lip} (\sigma)\right]^k  \E \left( \left|
			\int_r^t  \d s \int_{\R^d} \d y\int_{\R^d}d w \
			\mathcal{P}_{t-s}(y\,,w\,;x)
			|D_{r,z} u_n(s\,,y)|\, |D_{r,z} u_n(s\,,w)|  \bar{\bf f}(y-w) \right| ^{k/2} \right).
	\end{align*}
	Back-to-back appeals to the inequalities of Minkowski and Cauchy--Schwarz then leads us to the following:
	\begin{align*}
		& \E \left( \left|  \int_{(r,t)\times\R^d} \bm{p}_{t-s} (x-y)
			\Sigma_n  D_{r,z} u_n(s\,,y )\,\eta(\d s \,\d y)\right |^k \right)  \\
		& \le \left[ z_k {\rm Lip} (\sigma)\right]^k   \left[  \int_r^t \d s
			\int_{\R^d}  \d y \int_{\R^d}\d w\
			\mathcal{P}_{t-s}(y\,,w\,;x)
			\left\| D_{r,z} u_n(s\,,y) D_{r,z} u_n(s\,,w ) \right\|_{k/2}  \bar{\bf f}(y-w) \right] ^{k/2}\\
		 & \le \left[ z_k {\rm Lip} (\sigma)\right]^k   \left[  \int_r^t \d s
			\int_{\R^d}  \d y \int_{\R^d}\d w\
			\mathcal{P}_{t-s}(y\,,w\,;x)
			\left\| D_{r,z} u_n(s\,,y) \right\|_k \left\|D_{r,z} u_n(s\,,w ) \right\|_k  \bar{\bf f}(y-w) \right] ^{k/2}.
	\end{align*}	
	The preceding displayed computations yield the following inequality on the Malliavin derivative of
	$u_{n+1}(t\,,x)$:
	\begin{align*}
		&\|D_{r,z} u_{n+1}(t\,,x)\|_k\\
		&\le C_{T, k}\, \bm{p}_{t - r}(x - z) \\
		&+  z_k{\rm Lip} (\sigma)\left[
			\int_r^t  \d s \int_{\R^d}   \d y\int_{\R^d}\d w \
			\mathcal{P}_{t-s}(y\,,w\,;x)
			\left\| D_{r,z} u_n(s\,,y ) \| _k \|D_{r,z} u_n(s\,,w ) \right\|_k  \bar{\bf f}(y-w)\right]^{1/2}.
	\end{align*}
	
	By our induction hypothesis, ($P_n$) is valid; that is, \eqref{e2} holds [for $n$], whence
	\begin{equation}\label{W}\begin{split}
		& \int_r^t  \d s \int_{\R^d}   \d y\int_{\R^d}\d w \
			\mathcal{P}_{t-s}(y\,,w\,;x)
			\left\| D_{r,z} u_n(s\,,y ) \| _k \|D_{r,z} u_n(s\,,w ) \right\|_k  \bar{\bf f}(y-w)\\
		&\le  2C_{T, k}^2\sum_{i = 0}^{n}\gamma^i  \int_r^t   \d s\ h_i(s - r)
			\int_{\R^d} \d y\int_{\R^d}\d w\
			\mathcal{P}_{t-s}(y\,,w\,;x)
			\mathcal{P}_{t-s}(y\,,w\,;z) \bar{\bf f}(y-w).
	\end{split}\end{equation}
	We unscramble the notation for $\mathcal{P}$ in order to see that the
	elementary pointwise
	inequality,\footnote{See for example the formula below (2.10) in Chen and Huang \cite{CH19Comparison}.}
	\[
		\bm{p}_{t -s}(x - y)\bm{p}_s(y -z) = \bm{p}_t(x -z)\bm{p}_{s(t - s)/t}
		\left(y - z - \frac{s}{t}(x - z)\right),
	\]
	yields the following upper bound for the quantity on the right-hand side of \eqref{W}:
	\begin{align*}
		& 2C_{T, k}^2\bm{p}^2_{t- r}(x -z)\sum_{i = 0}^{n}\gamma^i
			\int_r^t   \d s\ h_i(s - r) \int_{\R^d} \d y \int_{\R^d}\d w  \  \bar{\bf f}(y-w) \\
		&\hskip.5in
			\times\bm{p}_{(s - r)(t - s)/(t-r)}\left(y - z - \frac{s - r}{t - r}(x - z)\right)
			\bm{p}_{(s - r)(t - s)/(t-r)}\left(w - z - \frac{s - r}{t - r}(x - z) \right) \\
		& = 2C_{T, k}^2\bm{p}^2_{t- r}(x -z)\sum_{i = 0}^{n}\gamma^i  \int_0^{t  - r}   \d s\
			h_i(s) \int_{\R^d} \d  y\   \bar{\bf f}(y) \bm{p}_{2s(t - r - s)/(t-r)}(y),
	\end{align*}
	where the final identity can be deduced from a change
	of variables $[y - w \to y]$ and the semigroup property of the heat kernel.
	
	Since every function $h_i$ is nondecreasing
	and $(a+b)^2\le2a^2+2b^2$ for all $a,b\ge0$, Lemma \ref{inequality:g} implies that
	\begin{align*}
		\|D_{r,z} u_{n+1}(t\,,x)\|_k^2 &\le  2C^2_{T, k}\, \bm{p}^2_{t - r}(x - z)
			+ 2^{(d+4)/2} \left[z_k {\rm Lip} (\sigma)\right]^2
			C_{T, k}^2\bm{p}^2_{t- r}(x -z) \\
		& \hskip2in\times \sum_{i = 0}^{n}\gamma^i\int_{\R^d} \d y\
			\bar{\bf f}(y)\int_0^{t - r}\d s\ h_i(s)\bm{p}_{2(t - r - s)}(y) \\
		&  = 2C^2_{T, k}\, \bm{p}^2_{t - r}(x - z)  + 2^{(d+4)/2}
			\left[z_k {\rm Lip} (\sigma)\right]^2
			C^2_{T, k}\bm{p}^2_{t- r}(x -z)
			\sum_{i = 0}^{n}\gamma^ih_{i + 1}(t - r) \\
		& =2C^2_{T, k}\, \bm{p}^2_{t - r}(x - z) \left(1 + \gamma\sum_{i = 0}^{n}
			\gamma^ih_{i + 1}(t - r)  \right).
	\end{align*}
	This proves the conditional validity of the proposition
	($P_{n+1}$), given that ($P_j$) is valid for all $j=0,\ldots,n$.
	Induction  yields \eqref{e2};
	we can now conclude the proof as follows.

	Because $\lim_{n\to\infty} u_n(t\,,x)= u(t\,,x)$ in $L^k(\Omega)$,
	\eqref{e2} and Lemma \ref{H:finite} together imply that
	\[
		\sup_{n\ge0}  \E \left( \| Du_n(t\,,x)\|^k_{\HH}\right) <\infty.
	\]
	Lemma 1.5.3 of Nualart \cite{Nualart} now implies that $u(t\,,x) \in \mathbb{D}^{1,k}$.
	
	Finally, it remains to show that the estimate \eqref{L^p_norm}
	holds for $u(t\,,x)$, where $t\in(0\,,T)$ and $x\in\R^d$
	are held fixed. This follows from the fact that $Du_n(t\,,x)$ converges
	in the weak topology of $L^k(\Omega\, ; \HH)$ to $Du(t\,,x)$
	possibly after moving to a subsequence.
	This proof is a little bit involved and carried out as follows:
	First note that, because of \eqref{e2},
	\[
		\sup_{n\ge0}  \E\left( \| Du_n(t\,,x)\|^k_{L^k(\R_+ \times \R^d)} \right) <\infty
		\qquad\text{for $1\le k<\frac2d+1$}.
	\]
	Fix such a $k$. It follows that, after possibly moving to subsequence,
	$Du_n(t\,,x)$  converges to $Du(t\,,x)$ in the weak topology of
	$L^k(\Omega\, ; \R_+\times \R^d)$,
	whence
	\[
		Du(t\,,x)\in L^k(\Omega \,; \R_+\times \R^d).
	\]
	Then, we use a smooth approximation $\{\psi_\epsilon\}_{\epsilon>0}$
	to the identity in $\R_+\times \R^d$,  and apply Fatou's lemma and
	duality for $L^p$-spaces in order
	to find that, for almost every $(s\,,y) \in [0\,,t] \times \R^d$ and for all $k\ge 2$,
	\begin{align*}
		\|D_{s,y}u(t\,,x) \|_k & \le  \sup_{\epsilon >0}
			\left \| \int_{\R_+\times \R^d} D_{s',y'} u(t\,,x)
			\psi_\epsilon(s'-s\,, y'-y) \d s'\d y' \right\|_k\\
	 & \le  \sup_{\epsilon>0}
		\sup_{\|G \|_{k/(k - 1)}\le 1}
		\left| \int_{\R_+\times \R^d}  \E\left[ G D_{s',y'} u(t\,,x) \right]
		\psi_\epsilon(s'-s\,, y'-y) \d s'\d y'  \right|.
	\end{align*}
	Choose and fix a random variable $G\in L^{k/(k-1)}(\Omega)$ such that
	$\E(|G|^{k/(k-1)})\le 1$. By \eqref{e2}, we can find a subsequence $n(1) < n(2)<\cdots$
	of positive integers such that
	\begin{align*}
		& \left| \int_{\R_+\times \R^d}  \E \left[ G D_{s',y'} u(t\,,x) \right]
			\psi_\epsilon(s'-s\,, y'-y) \d s'\d y'  \right|  \\
		&= \lim _{\ell\rightarrow \infty} \left|
			\int_{\R_+\times \R^d}  \E\left[ G D_{s',y'} u_{n(\ell)}(t\,,x) \right]
			\psi_\epsilon(s'-s\,, y'-y) \d s'\d y'   \right| \\
		&\le\limsup_{\ell\to\infty}
			\int_{\R_+\times \R^d}  \left\| D_{s',y'} u_{n(\ell)}(t\,,x) \right\|_k
			\psi_\epsilon(s'-s\,, y'-y) \d s'\d y' \\
		& \le 2C_{T, k}\,  \bm{p}_{t-s}(x-y)\left( \sum_{i=0}^\infty \gamma^i h_i(t-s)\right)^{1/2}.
	\end{align*}
	 An application of Lemma \ref{H:finite} completes the proof of Theorem \ref{malliavin_derivative}.
\end{proof}
	
\begin{remark}\label{rem:doob2}
	One can show that, for any fixed $(t\,,x) \in \R_+ \times \R^d$,  the mapping
	$(s\,,y) \mapsto D_{s,y} u(t\,,x)$
	from $(0\,,t) \times \R^d$ to $L^k(\Omega)$ is continuous for any $k\ge 2$.
	This follows from the fact that $D_{s,y} u(t\,,x)$ solves the following linear integral equation
	(see \cite[Proposition 3.2]{ChenHuang2019} when $f$ is a nonnegative function):
	\[
		D_{s,y} u(t\,,x)=  \bm{p}_{t-s} (x-y) \sigma(u(s\,,y)) +
		\int_{(s,t)\times\R^d} \bm{p}_{t-r} (x-z)
		\sigma'(u(r\,,z))  D_{s,y} u(r\,,z )\,\eta(\d r \,\d z),
	\]
	obtained by applying the operator $D$ to the equation satisfied by $u(t\,,x)$.
	In this context, the asserted $L^k(\Omega)$-continuity is proved by resorting
	to usual arguments based on the Burkholder-Davis-Gundy inequality.
	We omit the many details. This continuity property is relevant though,
	for it allows us to appeal to Doob's theory of separability \cite{Doob}
	in order to deduce the Lebesgue measurability of  $(s\,,y)\mapsto D_{s,y}u(t\,,x)$.
\end{remark}

\section{Proof of stationarity}\label{sec:stat}
For every $\varphi\in C(\R_+\times\R^d)$ and $y\in\R^d$ define shift operators
$\{\theta_y\}_{y\in\R^d}$ as follows:
\[
	(\varphi\circ \theta_y)(t\,,x) = \varphi(t\,,x+y).
\]
Clearly, $\theta:=\{\theta_y\}_{y\in\R^d}$ is a group under composition. The
following is used tacitly in the literature many times without explicit proof of even mention
(see for example \cite{CJK2013}). It also improves the assertion,  observed by Dalang \cite{Dalang1999} that the 2-point correlation function of $x\mapsto u(t\,,x)$
is invariant under $\theta$. When $\sigma(z)\propto z$
the latter moment invariance (and more) can be deduced directly
from an explicit Feynman--Kac type moment formula;
see for example Chen, Hu, and Nualart \cite{LCHN17}.

\begin{lemma}[Spatial Stationarity]\label{lem:stat}
	Suppose that $f$ either satisfies \eqref{Dalang}, or
	\eqref{f=h*h} for some $h\in\cup_{p>1}\mathcal{G}_p(\R^d)$, so that \eqref{SHE}
	has a unique random-field solution $u$; see Dalang \cite{Dalang1999} and Theorem \ref{th:exist}.
	Then, the random field $u\circ\theta_y$ has the same finite-dimensional
	distributions as $u$ for every $y\in\R^d$.
	In particular, for every $t\ge0$, the finite-dimensional
	distributions of $\{u(t\,,x+y)\}_{x\in\R^d}$ do not depend on $y\in\R^d$.
\end{lemma}

\begin{proof}
	The fact that \eqref{SHE} has a strong solution is another way to state
	that the transformation $\eta\mapsto u$ defines canonically a
	``solution map'' $\mathbb{S}$ via
	$u=\mathbb{S}(\eta)$, where we recall $\eta$ denotes the driving
	noise. Recall also that the generalized Gaussian random field $\eta$ can be identified with a
	densely-defined isonormal Gaussian
	process $C_c(\R_+\times\R^d)\ni \varphi\mapsto\eta(\varphi)$
	via Wiener integrals as follows:
	\[
		\eta(\varphi) = \int_{\R_+\times\R^d}\varphi\,\d\eta
		\qquad\text{for all $\varphi \in \HH$},
	\]
    where $\HH$ is the Hilbert space introduced in \S\ref{sec:Malliavin}.
	Since $C_c(\R_+\times\R)\ni \varphi\mapsto\eta(\varphi)\in L^2(\Omega)$
	is a continuous linear mapping, the preceding identifies $\eta$ completely provided only
	that we prescribe $\eta(\varphi)$ for every $\varphi\in C_c(\R_+\times\R)$.
	In this way, we can  define a Gaussian noise $\eta_y$---one for every $y\in\R^d$---via
	\begin{equation}\label{xi_h}
		\eta_y(\varphi) = \int_{\R_+\times\R^d} \varphi(t\,,x-y)\,\eta(\d t\,\d x)
		\qquad\text{for all $\varphi\in C_c(\R_+\times\R^d)$.}
	\end{equation}
	It is easy to check covariances in order to see that $\eta_y(\varphi)$
	and $\eta(\varphi)$ have the same law; therefore, the noises
	$\eta$ and $\eta_y$ have the same law for every $y\in\R^d$. Also,
	it follows from the construction of the Walsh/It\^o stochastic integral that
	for all $t\ge0$, $x,y\in\R^d$, and Walsh-integrable random fields $\Psi$,
	\begin{equation}\label{WI}
		\int_{(0,t)\times\R^d}\Psi(s\,,z-y)\eta(\d s\,\d z)
		= \int_{(0,t)\times\R^d}\Psi(s\,,z)\,\eta_y(\d s\,\d z)
		\qquad\text{a.s.}
	\end{equation}
	This can be proved by standard approximation arguments, using only the fact that
	\eqref{WI} holds by \eqref{xi_h} when $\Psi$ is a simple random field;
	see Walsh \cite[Chapter 2]{Walsh}.
	
	Finally, we may combine \eqref{SHE} and \eqref{WI} in order to see that
	for all $t\ge0$ and $x,y\in\R^d$,
	\begin{align*}
		u(t\,,x+y) &= 1 + \int_{(0,t)\times\R^d} \bm{p}_{t-s}(x+y-z)
			\sigma(u(s\,,z-y+y))\,\eta(\d s\,\d z)\\
		&=  1 + \int_{(0,t)\times\R^d} \bm{p}_{t-s}(x-z)
			\sigma(u(s\,,z+y))\,\eta_y(\d s\,\d z)
			\qquad\text{a.s.}
	\end{align*}
	This proves that $u\circ\theta_y = \mathbb{S}(\eta_y)$ a.s.\ for every $y\in\R^d$, where
	we recall $\mathbb{S}$ denotes the solution map in \eqref{SHE}.
	Because $u$ is continuous,
	the preceding is another way to state the
	first assertion of the result. The second assertion follows from the first for elementary
	reasons.
\end{proof}

Let us mention also the following simple fact.

\begin{lemma}\label{lem:Var:erg}
	A stationary process $Y:=\{Y(x)\}_{x\in\R^d}$ is ergodic provided that
	\begin{equation}\label{cond:var}
		\lim_{N\to\infty}
		\Var\left( \fint_{[0,N]^d} \prod_{j=1}^k g_j(Y(x+\zeta^j))\,\d x\right) =0,
	\end{equation}
	for all integers $k\ge1$, every $\zeta^1,\ldots,\zeta^k\in\R^d$,
	and all Lipschitz-continuous functions $g_1,\ldots,g_k:\R\to\R$ that satisfy \eqref{eq:WLOG}.
\end{lemma}

\begin{proof}
	Suppose $g_1,\ldots,g_k:\R\to\R$ are non-constant,
	Lipschitz-continuous functions, but do not necessarily satisfy
	\eqref{eq:WLOG}. We first verify that \eqref{cond:var} holds
	for these $g_i$'s as well. Indeed, define
	\[
		\widetilde{g}_j(w) := \frac{g_j(w)-g_j(0)}{\text{\rm Lip}(g_j)}
		\qquad\text{for all $j=1,\ldots,k$ and $w\in\R$},
	\]
	and observe that $\widetilde{g}_1,\ldots,\widetilde{g}_k:\R\to\R$
	satisfy \eqref{eq:WLOG}, and hence \eqref{cond:var} holds
	when we replace every $g_i$ with $\tilde{g}_i$. It is easy to see that
	\begin{equation}\label{VarVar}\begin{split}
		&\fint_{[0,N]^d} \prod_{j=1}^k g_j(Y(x+\zeta^j))\,\d x\\
		&\hskip1.5in= \sum_{E\subseteq\{1,\ldots,k\}}
			\prod_{l\in E} g_l(0) \fint_{[0,N]^d} \prod_{j\in\{1,\ldots,k\}\setminus E}
			\text{\rm Lip}(g_j)\,\widetilde{g}_j(Y(x+\zeta^j))\,\d x,
	\end{split}\end{equation}
	where a product over the empty set is identically defined as $1$.
	For example, when $k=2$, we have
	\begin{align*}
		&\fint_{[0,N]^d} g_1(Y(x+\zeta^1))g_2(Y(x+\zeta^2))\,\d x\\
		&\hskip1in= \fint_{[0,N]^d} \left[\text{\rm Lip}(g_1)\,\widetilde{g}_1(Y(x+\zeta^1))
			+ g_1(0)\right]\left[\text{\rm Lip}(g_2)\,\widetilde{g}_2(Y(x+\zeta^2))
			+ g_2(0)\right]\d x,
	\end{align*}
	which yields \eqref{VarVar} upon expanding the product inside the integral.
	
	Minkowski's inequality ensures that, for all random variables $X_1,\ldots,X_M\in L^2(\Omega)$,
	\[
		\Var(X_1+\cdots+X_M) \le \left(\sum_{i=1}^M \sqrt{\Var(X_i)}\right)^2
		\le M^2 \max_{1\le i\le M}\Var(X_i).
	\]
	Thus, we see from \eqref{VarVar} that
	\begin{align*}
		&\Var\left(\fint_{[0,N]^d} \prod_{j=1}^k g_j(Y(x+\zeta^j))\,\d x\right)\\
		&\le4^k  \max_{E\subseteq\{1,\ldots,k\}}
			\prod_{l\in E} g_l^2(0)\cdot
			\Var\left(\fint_{[0,N]^d} \prod_{j\in\{1,\ldots,k\}\setminus E}
			\text{\rm Lip}(g_j)\,\widetilde{g}_j(Y(x+\zeta^j))\,\d x\right)\\
		&\to0\quad\text{as $N\to\infty$},
	\end{align*}
	thanks to \eqref{cond:var}. This proves the assertion that if \eqref{cond:var}
	holds when $g_i$'s are Lipschitz and satisfy \eqref{eq:WLOG}, then
	\eqref{cond:var} continues to hold for non-constant, Lipschitz-continuous $g_i$'s, even when
	they do not satisfy \eqref{eq:WLOG}. And it is easy to see that ``non-constant'' can
	be removed from the latter assertion without changing its truth: We merely factor out
	of the variance the constant $g_i$'s, and relabel the remaining $g_j$'s, thus reducing the problem
	to the non-constant case.
	
	We now apply the preceding with $g_i$'s replaced with sines and cosines, in order
	to deduce from   stationarity that
	\[
		\lim_{N\to\infty}
		\fint_{[0,N]^d}\exp\left\{i\sum_{j=1}^k z_jY(x+\zeta^j)\right\}\d x
		=\E\left[\exp\left\{i\sum_{j=1}^kz_jY(\zeta^j)\right\}
		\right]\qquad\text{in $L^2(\Omega)$},
	\]
	for all $z_1,\ldots,z_k\in\R$ and $\zeta^1,\ldots,\zeta^k\in\R^d$.
	On the other hand, von-Neumann's $L^2$ version of the
	ergodic theorem \cite{Peterson} tells us that
	\[
		\lim_{N\to\infty}
		\fint_{[0,N]^d}\exp\left\{i\sum_{j=1}^k z_jY(x+\zeta^j)\right\}\d x
		= \E\left[ \left. \exp \left\{i\sum_{j=1}^k z_jY(\zeta^j)\right\}
		\ \right|\, \mathcal{I}\right]
		\qquad\text{in $L^2(\Omega)$,}
	\]
	where $\mathcal{I}$ denotes the invariant $\sigma$-algebra
	of $Y$. Equate the preceding two displays, and apply the inversion theorem of
	Fourier transforms, in order to see that
	every random vector of the form
	$(Y(\zeta^1)\,,\ldots,Y(\zeta^k))$ is independent of $\mathcal{I}$.
	This implies that $\mathcal{I}$ is independent of the $\sigma$-algebra generated
	by $Y$, and in particular $\mathcal{I}$ is independent of itself. This in turn
	proves the result.
\end{proof}

\section{Proofs of Theorem \ref{th:1} (part 1),
	Theorems \ref{th:main:intro} and \ref{th:Poincare}, and Corollary \ref{co:main:intro}}\label{sec:Poincare}

We are ready to begin the proof of the Poincar\'e-type inequalities
of Theorems \ref{th:Poincare} and \ref{th:Poincare:bis}. Then we will show that,
among other things, our Poincar\'e-type inequalities imply the desired
spatial ergodicity of $u$.

\begin{proof}[Proof of Theorems \ref{th:Poincare} and \ref{th:Poincare:bis}]
	We shall prove Theorems \ref{th:Poincare} and \ref{th:Poincare:bis}
	at once, since one argument follows the other after we make small adjustments.
	
	Define
	\begin{equation}\label{V:G}
		V_N := {\rm Var} \left (\fint_{[0,N]^d}  \mathcal{G}(x) \d x \right)
		\quad\text{and}\quad
		\mathcal{G}(x) :=
		\prod_{j=1}^k g_j ( u(t\,,x+\zeta^j))\qquad
		\text{for all $x\in\R^d$},
	\end{equation}
	so that $V_N =
	\fint_{[0, N]^{2d}}  {\rm Cov} (  \mathcal{G}(x) \,, \mathcal{G}(y))\, \d x \d y.$
	We plan to calculate ${\rm Cov} (  \mathcal{G}(x)\,, \mathcal{G}(y))$,
	pointwise, using the Clark--Ocone formula (Proposition \ref{Clark-Ocone}). To
	this end, we apply the  chain rule
	for the Malliavin derivative \cite[Proposition 1.2.4]{Nualart}
	in order to see that
	\[
		D_{s,z} \mathcal{G}(x) = \sum_{j_0=1}^k
		\left(\prod_{\substack{j=1\\j \not= j_0}}^{k}
		g_j \left( u(t\,,x+\zeta^j) \right) \right)
		g_{j_0}' \left(u(t\,,x+\zeta^{j_0})\right)
		D_{s,z} u(t\,, x+ \zeta^{j_0}).
	\]
	The covariance structure of $\eta$ [see \eqref{Cov:eta}]
	and Proposition \ref{Clark-Ocone} together ensure that,
	when $f$ is additionally a \emph{function},
	\begin{align*}
		\left| {\rm Cov} (  \mathcal{G}(x)\,, \mathcal{G}(y)) \right|
			&=  \left|\int_0^t \d s\int_{\R^d}\d z\int_{\R^d}\d w\
			\E \left\{ \E \left( D_{s,z} \mathcal{G}(x) \mid \mathcal{F}_s \right)
			\cdot
			\E \left( D_{s,w} \mathcal{G}(y) \mid \mathcal{F}_s \right) \right\}  f(z - w) \right| \\
		&\le  \int_0^t \d s\int_{\R^d}\d z\int_{\R^d}\d w\
			\left\| D_{s,z} \mathcal{G}(x)  \right\|_2
			\left\| D_{s,w} \mathcal{G}(y) \right\|_2
			\bar{\bf f}(z - w),
	\end{align*}
	whence
	\[
		V_N \le \fint_{[0,N]^{2d}}\d x\d y
		\int_0^t \d s\int_{\R^d}\d z\int_{\R^d}\d w\
		\left\| D_{s,z} \mathcal{G}(x)  \right\|_2
		\left\| D_{s,w} \mathcal{G}(y) \right\|_2
		\bar{\bf f}(z - w),
	\]
	where $\bar{\bf f}$ is defined in the statement of Theorem \ref{malliavin_derivative},
	and is a \emph{function}. If $f$ is a measure, then we can adapt the preceding. Since we have
	said this sort of thing before in this paper, without mentioning how to adapt, we make the adaptation
	to the measure case now by merely observing that, in general, the preceding covariance bound
	gets adapted to the following, and for the same reasons as above:
	\begin{equation}\label{eq:V_N}
		V_N\le\fint_{[0,N]^{2d}}\d x\d y
		\int_0^t\left\langle \psi(s\,,x\,,\bullet) \,, \psi(s\,,y\,,\bullet)*\bar{\bf f}
		\right\rangle_{L^2(\R^d)}\,\d s,
	\end{equation}
	where $\psi(s\,,x\,,z):= \|D_{s,z}\mathcal{G}(x)\|_2$.
	In any case, Theorems \ref{th:exist} and \ref{malliavin_derivative} together imply
	the existence of a real number $c=c(T\,,k)$ such that
	\begin{equation}\label{psi:UB}\begin{split}
		\| D_{s,z} \mathcal{G}(x) \|_2
			&\le  \sum_{j_0=1}^k  \left(\prod_{j=1, j \not= j_0}^{k}
			\|  g_j ( u(t\,,x+\zeta^j))  \|_{2k} \right)  \| D_{s,z} u(t\,, x+ \zeta^{j_0})\|_{2k} \\
		&\le c\sum_{j=1}^k  \bm{p}_{t-s} (x+ \zeta^{j} -z),
	\end{split}\end{equation}
	uniformly for all $0<s<t\le T$ and $x,z\in\R^d$. Recall the probability density
	function $I_N$ from \eqref{I_N}. The preceding  can be now combined with the Tonelli
	theorem and the semigroup property of the heat kernel in order to yield
	\begin{align*}
		V_N&\le c^2\sum_{j,\ell=1}^k
			\fint_{[0, N]^{2d}}\d x\d y   \int_0^t \d s \
			\left\langle \bm{p}_{t-s}(x+\zeta^j-\bullet)\,,
			\bm{p}_{t-s}(y+\zeta^\ell-\bullet)*\bar{\bf f}\right\rangle_{L^2(\R^d)}\\
		&= c^2\sum_{j,\ell=1}^k
			\int_0^t \left( \bm{p}_{2(t-s)} * I_N * \tilde{I}_N *\bar{\bf f}\right)
			\left( \zeta^j-\zeta^\ell\right)\d s.
	\end{align*}
	Since $I_N*\tilde{I}_N\in C_c(\R^d)$ and $I_N*\tilde{I}_N$ is
	nonnegative definite, the function
	$I_{N}*\tilde{I}_{N}*\bar{\bf f}$ is continuous and nonnegative-definite,
	whence also maximized at $0$. Because the total integral of $\bm{p}_{2(t-s)}$ is one,
	it follows that
	\[
		V_N\le
		c^2k^2t  (  I_N * \tilde{I}_N *\bar{\bf f})(0)
		\le c^2k^2t \,\frac{\bar{\bf f}( [-N\,,N]^d)}{N^d};
	\]
	see \eqref{III} for the last inequality. This completes the proof of
	\eqref{eq:Poincare} when $f$ satisfies Dalang's condition \eqref{Dalang},
	as well as the proof of \eqref{eq:Poincare:bis} when $f$ satisfies \eqref{f=h*h}
	for some $h\in\cup_{p>1}\mathcal{G}_p(\R^d)$.
\end{proof}

We can now prove the remaining results from the Introduction.

\begin{proof}[Proof of part 1 of Theorem \ref{th:1}]
	Parts 2 and 3 of Theorem \ref{th:1} were proved respectively in
	\S\ref{sec:harmonic_3.3} and \S\ref{sec:part3}. We now conclude the proof
	of Theorem \ref{th:1} by verifying its first part. With this in mind, suppose
	$f$ satisfies \eqref{Dalang} and $\hat{f}\{0\}=0$, equivalently,
	\[
		\lim_{N\to\infty} \frac{f\left( [-N\,,N]^d\right)}{N^d}=0,
	\]
	thanks to part 2 of Theorem \ref{th:1}, which has already been established.
	According to the above hypothesis and Theorem \ref{th:Poincare},
	\begin{equation}\label{Var->0}
		\lim_{N\to\infty}\Var\left( \fint_{[0,N]^d}
		\prod_{j=1}^k g_j\left( u(t\,,x+\zeta^j)\right)\d x\right)=0,
	\end{equation}
	for all $t>0$, $\zeta^1,\ldots,\zeta^k\in\R^d$,
	and all Lipschitz functions $g_1,\ldots,g_k:\R\to\R$ that satisfy \eqref{eq:WLOG}.
	Lemma \ref{lem:Var:erg} now implies that $u$ is spatially ergodic, and
	concludes the proof of part 1 of Theorem \ref{th:1}.
\end{proof}

\begin{proof}[Proof of Theorem \ref{th:main:intro}]
	As was the case also in the  proof of Theorem \ref{th:1},
	the asserted stationarity of the solution has been proved earlier in Lemma \ref{lem:stat}.
	Now suppose $f$ satisfies \eqref{f=h*h}
	for some $h\in\cup_{p>1}\mathcal{G}_p(\R^d)$. Proposition \ref{pr:PD}
	tells us that $|h|*|\tilde{h}|$ is a function of positive type; thus, it vanishes
	at infinity among other things. This immediately yields
	\[
		\lim_{N \rightarrow \infty}
		\fint_{[-N, N]^d}\left(|h|*|\tilde{h}|\right)(x)\,\d x = 0,
	\]
	and hence \eqref{Var->0} (see Theorem \ref{th:Poincare:bis}). An appeal to
	Lemma \ref{lem:Var:erg} ends the proof of Theorem \ref{th:main:intro}.
\end{proof}

Finally, we verify Corollary \ref{co:main:intro}.
The proof is elementary. We include it here
however since the proof depends crucially on careful computation of the various exponents
in \eqref{h:1}--\eqref{exponent:alpha} below.

\begin{proof}[Proof of Corollary \ref{co:main:intro}]
	If $h\in L^2(\R^d)$ then we set $p=q=2$ to see that
	$h\in L^p_{\text{\it loc}}(\R^d)$ and
	\[
		\int_0^1\left( \|h\|_{L^p(\mathbb{B}_r)} \|h\|_{L^q(\mathbb{B}_r^c)}
		+ \|h\|_{L^2(\mathbb{B}_r^c)}^2\right)\omega_d(r)\,\d r
		\le 2\|h\|_{L^2(\R^d)}^2\int_0^1\omega_d(r)\,\d r,
	\]
	so that \eqref{cond:omega} holds thanks to the local integrability of $\omega_d$.
	Thus, it remains to assume that \eqref{cond:co:main:intro} holds. In that case,
	we appeal to \eqref{cond:co:main:intro} and integrate in spherical coordinates
	in order to see that
	\[
		\int_{\mathbb{B}_r}|h(x)|^p\,\d x \lesssim
		\int_0^r s^{d-1 - p(d+\alpha)/2}\,\d s\quad\text{%
		simultaneously for every $r\in(0\,,1)$.}
	\]
	Hence,
	\[
		h\in L^p_{\textit{\it loc}}(\R^d)\quad\text{iff}\quad
		p<\frac{2d}{d+\alpha}.
	\]
	Since $\alpha<d$, it follows that $2d/(d+\alpha)>1$
	and hence $h\in L^p_{\text{\it loc}}(\R^d)$ for every $p$ between
	$1$ and $2d/(d+\alpha)$.
	For every such $p$, \eqref{cond:co:main:intro} ensures that
	\begin{equation}\label{h:1}
		\|h\|_{L^p(\mathbb{B}_r)} \lesssim r^{(d/p)-(d+\alpha)/2}\quad\text{%
		simultaneously for every $r\in(0\,,1)$.}
	\end{equation}
	Choose one such $p$ and define $q:=p/(p-1)$, so that $p^{-1}+q^{-1}=1$.
	Eq.\ \eqref{cond:co:main:intro} implies that, for every $r\in(0\,,1)$,
	\begin{align*}
		\int_{\mathbb{B}_r^c} |h(x)|^q\,\d x
			&\le \int_{r<\|x\|<1}|h(x)|^q\,\d x + \int_{\|x\|>1}|h(x)|^q\,\d x\\
		&\lesssim\int_r^1 t^{d-1-q(d+\alpha)/2}\,\d t + \int_1^\infty t^{d-1-q(d+\beta)/2}\,\d t,
	\end{align*}
	where the implied constants do not depend on $r\in(0\,,1)$.
	The first integral is convergent regardless of the choice of $p$ (hence also $q$). The second integral
	converges iff
	\begin{equation}\label{qd}
		q > \frac{2d}{d+\beta},
	\end{equation}
	which can certainly be arranged if $p$ were chosen sufficiently
	close to $1$.\footnote{%
		To be concrete, we may select $1< p < d/(d-1)$
		to ensure that $q>d$, so that \eqref{qd} holds.
		}
	Choose and fix $p>1$ sufficiently close to $1$ in order to ensure that
	\eqref{qd} holds, whence
	\begin{equation}\label{h:2}
		\|h\|_{L^q(\mathbb{B}_r^c)} \lesssim r^{(d/q)-(d+\alpha)/2}\quad\text{%
		simultaneously for every $r\in(0\,,1)$.}
	\end{equation}
	Finally, we may repeat
	the preceding with $q$ replaced everywhere with $2$ in order to see that
	\begin{equation}\label{h:3}
		\|h\|_{L^2(\mathbb{B}_r^c)}
		\lesssim r^{-\alpha/2}\quad\text{%
		simultaneously for every $r\in(0\,,1)$.}
	\end{equation}
	We may now combine \eqref{h:1}, \eqref{h:2}, and \eqref{h:3}
	in order to see that,
	\begin{align}\label{exponent:alpha}
		\|h\|_{L^p(\mathbb{B}_r)}\|h\|_{L^q(\mathbb{B}_r^c)} +
		\|h\|_{L^2(\mathbb{B}_r^c)}^2 \lesssim r^{-\alpha}\quad\text{%
		simultaneously for every $r\in(0\,,1)$.}
	\end{align}
	Because $\alpha<2\wedge d$, it follows that
	$h\in\mathcal{G}_p(\R^d)$ for all $p$ sufficiently close to $1$.
\end{proof}

\section{Applications}\label{sec:Appl}

The Poincar\'e-type inequalities of Theorems \ref{th:Poincare} and \ref{th:Poincare:bis}
have many consequences other than those mentioned in Theorems \ref{th:1}
and \ref{th:main:intro}. We conclude the paper by presenting two rather different
applications of these Poincar\'e-type inequalities.

\subsection{Spatial mixing}\label{subsec:mixing}
We say that $u$ is \emph{spatially mixing} if the random field
$u(t)$ is (weakly) mixing for every $t>0$ \cite{DymMcKean,Maruyama,Peterson}.
Recall that this means that
\begin{equation}\label{mixing}
	\lim_{\|x\|\to\infty}\Cov \left[ \prod_{j=1}^k g_j\left( u(t\,,x + \zeta^{j}) \right) ~,~
	\prod_{l=1}^k g_l\left( u(t\,,\zeta^{l})\right) \right] =0,
\end{equation}
for all integers $k\ge 1$, real numbers $t>0$, $\zeta^1,\ldots,\zeta^k\in\R^d$,
and functions $g_1,\ldots,g_k$ of the form
$g_j(w) = \bm{1}_{(-\infty,a_j]}(w)$ for $w\in\R$
and arbitrary $a_1,\ldots,a_k\in\R$. Our next result finds  unimprovable conditions for
spatial mixing of the solution to \eqref{SHE}. When $d=1$ and $\sigma\equiv$ constant,
our condition is  sharp, and in agreement with classical results of Maruyama
\cite{Maruyama} on mixing properties of stationary Gaussian processes.

\begin{corollary}\label{mix1}
	Suppose $f$ satisfies Dalang's condition \eqref{Dalang}.
	Then, $u$ is spatially mixing if
	\begin{equation}\label{cond:mix}
		\lim_{\|x\|\to\infty}\left( \bm{v}_\lambda*f\right)(x)=0,
		\quad\text{equivalently\ \ if}\quad
		\lim_{\|x\|\to\infty}
		\int_{\R^d}\frac{\e^{ix\cdot z}}{2\lambda+\|z\|^2}\,\hat{f}(\d z)=0,
	\end{equation}
	for some, hence all, $\lambda>0$. Moreover, \eqref{cond:mix} is a necessary and sufficient condition for
	the spatial mixing of $u$ in the case that $\sigma$ is a constant.
\end{corollary}

We pause and briefly examine condition \eqref{cond:mix} before we prove the corollary.

\begin{example}
	If the spectral measure $\hat{f}$ is a function,
	then Dalang's condition \eqref{Dalang} and the classical
	Riemann--Lebesgue lemma of Fourier analysis together guarantee that the second
	formulation in condition \eqref{cond:mix}
	holds. Thus, $u$ is spatially mixing whenever
	the underlying noise has a spectral density that satisfies Dalang's condition.
\end{example}

\begin{example}
	If $f$ is a function that satisfies Dalang's condition \eqref{Dalang} as well as parts
	1 and 2 of Definition \ref{def:PT}, then the proof of our next corollary can be
	easily adapted\footnote{Basically, one replaces the function $|h|*|\tilde{h}|$
	everywhere in the proof of Corollary \ref{mix2}
	by the function $f$.} in order to prove that the first condition
	in \eqref{cond:mix} holds. In particular, $u$ is spatially mixing
	provided that the correlation $f$
	is a function of positive type that satisfies Dalang's condition; and
	in fact condition
	3 of Definition \ref{def:PT} is not needed for mixing to hold.
\end{example}

\begin{proof}[Proof of Corollary \ref{mix1}]
	We can approximate every
	$\bm{1}_{(-\infty,a_j]}$ from above and below by Lipschitz-continuous
	functions in order to see that $u$ is spatially
	mixing if and only if \eqref{mixing} holds for all $k\ge 1$,
	real numbers $t>0$, $\zeta^1,\ldots,\zeta^k\in\R^d$,
	and Lipschitz-continuous functions $g_1,\ldots,g_k:\R\to\R.$
	In other words, it suffices to prove that
	\begin{equation}\label{enough}
		\lim_{\|x\|\to\infty}\Cov\left( \mathcal{G}(x)\,,\mathcal{G}(0)\right)=0,
	\end{equation}
	where $\mathcal{G}$ is the random field was defined in \eqref{V:G},
	and where the functions $g_1,\ldots,g_k$ therein are Lipschitz continuous.
	We may, and will, assume further and without loss in generality that
	$g_1,\ldots,g_k$ satisfy \eqref{eq:WLOG}. This can be justified using
	an argument that appeared earlier in the proof of
	Lemma \ref{lem:Var:erg}.
	
	We now use Theorem \ref{th:Poincare}, in exactly the same manner that
	was used to derive \eqref{eq:V_N}, in order to find that for every $x\in\R^d$,
	and for $\psi(s\,,x\,,z) := \|D_{s,z}\mathcal{G}(x)\|_2$,
	\begin{align*}
		\Cov\left( \mathcal{G}(x)\,,\mathcal{G}(0)\right) &\le
			\int_0^t\left\langle\psi(s\,,x\,,\bullet)\,,
			\psi(s\,,0\,,\bullet)*f\right\rangle_{L^2(\R^d)}\d s\\
		&\le c^2\sum_{j,\ell=1}^k\int_0^t\left\langle \bm{p}_{t-s}(x+\zeta^j-\bullet)
			\,,\bm{p}_{t-s}(\zeta^\ell-\bullet) * f\right\rangle_{L^2(\R^d)}\d s,
	\end{align*}
	for the same constant $c>0$ that appeared in \eqref{psi:UB}. The semigroup
	property of the heat kernel now yields
	\begin{align*}
		\Cov\left( \mathcal{G}(x)\,,\mathcal{G}(0)\right)
			&\le c^2\sum_{j,\ell=1}^k\int_0^t\left( \bm{p}_{2s}*f\right)
			\left(x+\zeta^j-\zeta^\ell\right)\d s\\
		&\le c^2\e^{2\lambda t}\sum_{j,\ell=1}^k\int_0^t
			\e^{-2\lambda s}\left( \bm{p}_{2s}*f\right)
			\left(x+\zeta^j-\zeta^\ell\right)\d s\\
		&\le \frac{c^2\e^{2\lambda t}}{2}\sum_{j,\ell=1}^k\left(
			\bm{v}_\lambda*f\right)\left(x+\zeta^j-\zeta^\ell\right).
	\end{align*}
	This demonstrates that the first condition in \eqref{cond:mix}  ensures \eqref{enough},
	and completes the proof of spatial mixing of $u$. Next, we verify that
	the two conditions in \eqref{cond:mix} are equivalent.
	
	Because $\bm{p}_s\in\mathscr{S}(\R^d)$ for every $s>0$,
	\[
		( \bm{p}_s*f)(x) = \frac{1}{(2\pi)^d}\int_{\R^d} \e^{ix\cdot z-s\|z\|^2/2}\hat{f}(\d z)
		\qquad\text{for all $x\in\R^d$}.
	\]
	Multiply both sides by $\exp(-\lambda s)$ and integrate $[\d s]$ to find that
	\[
		(\bm{v}_\lambda*f)(x) \propto \int_{\R^d}
		\frac{\e^{ix\cdot z}}{2\lambda+\|z\|^2}\,\hat{f}(\d z)
		\qquad\text{for all $\lambda>0$ and $x\in\R^d$}.
	\]
	
	In order to complete the proof, suppose $\sigma\equiv c_0$ for some
	$c_0>0$, and assume that $u$ is spatially mixing. Among other things,
	we can specialize \eqref{enough} to deduce that for every $t>0$,
	\begin{equation}\label{Cov-u}
		\lim_{\|x\|\to\infty}\Cov\left[ u(t\,,x) \,, u(t\,,0)\right]=0.
	\end{equation}
	But \eqref{addive_noise} and  \eqref{Cov:eta} together imply that
	for every $x\in\R^d$ and $t,\lambda>0$,
	\begin{align*}
		\Cov\left[ u(t\,,x) \,, u(t\,,0)\right] &= c_0^2
			\int_0^t\left\langle \bm{p}_s(x+\bullet)\,,\bm{p}_s*f\right\rangle_{L^2(\R^d)}\d s\\
		&= c_0^2\int_0^t \left( \bm{p}_{2s}*f\right)(x)\,\d s
			\ge c_0^2\int_0^t \e^{-\lambda s}\left( \bm{p}_{2s}*f\right)(x)\,\d s.
	\end{align*}
	On one hand, the above and \eqref{Cov-u} together
	tell us that, for every $t>0$,
	\begin{equation}\label{cr}
		\int_0^t \e^{-\lambda s}\left(\bm{p}_{2s}*f\right)(x)\,\d s\to 0
		\qquad\text{as $\|x\|\to\infty$}.
	\end{equation}
	On the other hand,
	\begin{align*}
		\int_t^\infty \e^{-\lambda s}\left(\bm{p}_{2s}*f\right)(x)\,\d s
			&= \frac{1}{(2\pi)^d}\int_t^\infty \e^{-\lambda s}\,\d s
			\int_{\R^d}\hat{f}(\d z)\ \e^{ix\cdot z-s\|z\|^2}\\
		&\propto\int_{\R^d}
			\frac{\e^{ix\cdot z - t(\lambda  +\|z\|^2)}}{\lambda  + \|z\|^2}
			\,\hat{f}(\d z),
	\end{align*}
	which leads to the following crude bound, valid uniformly for all $t>0$:
	\[
		\int_t^\infty \e^{-\lambda s}\left( \bm{p}_{2s}*f\right)(x)\,\d s
		\lesssim\e^{-\lambda t}\int_{\R^d}\frac{\hat{f}(\d z)}{\lambda +\|z\|^2}.
	\]
	Combine this bound with \eqref{cr} in order to see that
	\[
		\limsup_{\|x\|\to\infty}\int_0^\infty\e^{-\lambda s}\left( \bm{p}_{2s}*f\right)(x)\,\d s
		\lesssim\e^{-\lambda t}\int_{\R^d}\frac{\hat{f}(\d z)}{\lambda +\|z\|^2}
		\qquad\text{for every $t>0$}.
	\]
	Let $t\to\infty$ and appeal to Dalang's condition \eqref{Dalang}
	in order to see that
	the left-hand side is zero for every $\lambda>0$. This concludes the proof.
\end{proof}

\begin{corollary}\label{mix2}
	If $f$ satisfies \eqref{f=h*h} for some $h\in\cup_{p>1}\mathcal{G}_p(\R^d)$,
	then $u$ is spatially mixing.
\end{corollary}

%One can estimate the preceding integral using Proposition \ref{pr:PD}. We will
%not work out the details.

\begin{proof}
	We go careful over the proof of the preceding and adapt
	using Theorem \ref{th:Poincare:bis} instead of
	Theorem \ref{th:Poincare} (as we have in the transition
	from the proof of Theorem \ref{th:1} to that of Theorem \ref{th:main:intro}) in order to
	see that Corollary \ref{mix2} follows once we can establish that
	\begin{equation}\label{game}
		\lim_{\|x\|\to\infty} \int_{\R^d}\bm{v}_1(y)\left( |h|*|\tilde{h}|\right)(x-y)\,\d y
		=0.
	\end{equation}
	On one hand,
	since $|h|*|\tilde{h}|$ vanishes uniformly at infinity [Proposition
	\ref{pr:PD}] and $\bm{v}_1$ is a probability density function,
	\[
		\int_{\|y\|<\|x\|/2}\bm{v}_1(y)\left( |h|*|\tilde{h}|\right)(x-y)\,\d y
		\le\sup_{\|w\|>\|x\|/2}\left( |h|*|\tilde{h}|\right)(w)=o(1)
		\qquad\text{as $\|x\|\to\infty$}.
	\]
	On the other hand, a similar argument shows that
	\begin{align*}
		&\int_{\|y\|>\|x\|/2}\bm{v}_1(y)\left( |h|*|\tilde{h}|\right)(x-y)\,\d y\\
		&= \int_{\substack{\|y\|>\|x\|/2\\\|y-x\|<\|x\|/2}}\bm{v}_1(y)
			\left( |h|*|\tilde{h}|\right)(x-y)\,\d y
			+ \int_{\substack{\|y\|>\|x\|/2\\\|y-x\|>\|x\|/2}}\bm{v}_1(y)
			\left( |h|*|\tilde{h}|\right)(x-y)\,\d y\\
		&= \int_{\substack{\|y\|>\|x\|/2\\\|y-x\|<\|x\|/2}}\bm{v}_1(y)
			\left( |h|*|\tilde{h}|\right)(x-y)\,\d y
			+ o(1)\qquad\text{as $\|x\|\to\infty$}.
	\end{align*}
	Let $\Xi(x)$ denote the final integral in the above. It remains to prove that
	$\Xi(x)\to0$ as $\|x\|\to\infty$.
	Since $\bm{v}_1(y)$ decreases monotonically as $\|y\|$
	increases,
	\[
		\Xi(x)\le \mathcal{V}(\|x\|/2)\int_{\|y\|<\|x\|/2}
		\left( |h|*|\tilde{h}|\right)(y)\,\d y,
	\]
	where
	\[
		\mathcal{V}(a) = \frac{1}{(2\pi)^{d/2}}\int_0^\infty
		s^{-d/2}\exp\left( -s-\frac{a^2}{2s}\right)\d s
		\propto a^{-(d-2)/2} K_{(d-2)/2}\left(
		a\sqrt 2\right),
	\]
	for all $a>0$ and $K_\nu:=$ modified Bessel function of the second kind.
	%(see \cite[10.32.10, p.\ XXX]{NIST} for example).
	Elementary asymptotic evaluations imply that
	$K_{(d-2)/2}(a)\lesssim a^{-1/2}\exp(-a)$ for all $a>1$ (see \cite[10.25.3]{OLBC10}),
	whence $\mathcal{V}(a)\lesssim a^{-(d-1)/2}\exp(-a\sqrt 2)$.
	Consequently,
	\[
		\Xi(x)\lesssim \e^{-\|x\|\sqrt 2}\int_{\|y\|<\|x\|/2}
		\left( |h|*|\tilde{h}|\right)(y)\,\d y,
	\]
	uniformly for all $x\in\mathbb{B}_2^c$ (with room to spare).
	Because $c := \sup_{y\in\mathbb{B}_1^c}(|h|*|\tilde{h}|)(y)$ is finite
	(see Proposition \ref{pr:PD}),
	\[
		\int_{\|y\|<\|x\|/2}
		\left( |h|*|\tilde{h}|\right)(y)\,\d y
		\le \int_{\mathbb{B}_1}
		\left( |h|*|\tilde{h}|\right)(y)\,\d y +c\int_{1<\|y\|<\|x\|/2}\d y
		\lesssim \|x\|^d,
	\]
	as $\|x\|\to\infty$. This proves that
	$\Xi(x) \lesssim \|x\|^d \exp(-\|x\|\sqrt 2) = o(1),$
	as $\|x\|\to\infty$, which completes the proof.
\end{proof}

\subsection{Intermittency}\label{subsec:interm}

In this final section we include an additional application of our
Poincar\'e inequalities. In order to simplify the exposition, we
consider  \eqref{SHE} in the case of the \emph{parabolic Anderson model}
driven by space-time white noise.
That is, we propose to study the SPDE,
\begin{equation}\label{PAM}
	\partial_t u = \tfrac12\partial^2_x u + u\eta\qquad\text{%
	on $(0\,,\infty)\times\R$},
\end{equation}
subject to $u(0)\equiv1$, where
\[
	\E[\eta(t\,,x)\eta(s\,,y)] = \delta_0(t-s)\delta_0(x-y)
	\qquad\text{for every $s,t>0$ and $x,y\in\R$}.
\]

It is well known that $u(t\,,x)>0$ for all $t>0$ and $x\in\R$ off a single $\P$-null set
(see Mueller \cite{Mueller1991,Mueller2009}),
and that the solution is unbounded at all times $t>0$, viz.,
\begin{equation}\label{blowup}
	\lim_{N\to\infty} \sup_{x\in[0,N]}
	u(t\,,x)=\infty\qquad\text{a.s.\ for every $t>0$}.
\end{equation}
In fact, Chen \cite[Theorem 1.7]{XC16} has
established the following improvement of \eqref{blowup}:
\begin{equation}\label{sup}
	\lim_{N\to\infty}\sup_{x\in[0,N]}\frac{\log u(t\,,x)}{(\log N)^{2/3}}
	=  \frac34\left(\frac{2t}{3}\right)^{2/3}\qquad\text{a.s.}\footnote{Chen \cite{XC16}
	proves this fact with $\sup_{x\in[-N,N]}u(t\,,x)$ in place of
	$\sup_{x\in[0,N]}u(t\,,x)$. The
	present statement is proved in the same way, however.}
\end{equation}

Conus et al \cite{CJK12} have studied the Lebesgue measure
of the set of $x\in[0\,,N]$ where
$u(t\,,x)$ is almost as tall as the maximum possible, as given in \eqref{sup}.
The following verifies one of their conjectures;
see \cite[see  (1.5)]{CJK12}.

\begin{corollary}\label{co:Islands}
	Choose and fix some $t>0$, and define
	$d(\alpha) := 4\alpha 3^{-3/2}\sqrt{6/t}$ for all $\alpha>0$.
	Whenever $d(\alpha)<1/2$, the following holds almost surely:
	\begin{equation}\label{Dim}
		\lim_{N\to\infty}
		\frac{1}{\log N}\log\left(\int_0^N \bm{1}_{\{ u(t,x) >
		\exp[(\alpha\log N)^{2/3}]\}}\,\d x\right)
		= 1 - d(\alpha).
	\end{equation}
\end{corollary}

The quantity on the left-hand side of \eqref{Dim} is a kind of ``macroscopic fractal dimension''
for the set $\mathscr{P}(\theta)$ of $x\in\R$ such that $u(t\,,x)$ exceeds
$\exp\{\theta(\log |x|)^{2/3}\}$.
In this way we can see that the ``fractal dimension formula'' \eqref{Dim} yields a ``codimension
formula'' for the
macroscopic Hausdorff dimension of $\mathscr{P}(\theta)$; this
should be compared to the related dimension formulas of Khoshnevisan, Kim,
and Xiao \cite[Theorem 1.2]{KKX17}.

\begin{proof}
	Before we begin,  let us observe that Theorem \ref{th:Poincare} implies that\footnote{%
		Theorem \ref{th:Poincare} requires also that $g(0)=0$. We obtain \eqref{eq:PP} by
		replacing $g$ by $g-g(0)$, without altering the value of the variance.
	}
	\begin{equation}\label{eq:PP}
		\sup_{g:\, \text{\rm Lip}(g)\le 1}\,\sup_{N>0}
		\Var\left( \frac{1}{\sqrt N}\int_0^N g(u(t\,,x))\,\d x\right) <\infty.
	\end{equation}
	Now we choose and fix $\alpha>0$, and define $a_N := \exp\{ (\alpha\log_+ N)^{2/3}\}$
	for every $N>0$.
	We plan to apply \eqref{eq:PP} with $g$ replaced by either $g_N$ or $G_N$,
	where
	\[
		G_N (z) := 1\wedge (z - a_N +1)_+
		\quad\text{and}\quad
		g_N (z) := 1\wedge (z - a_N)_+.
	\]
	According to Theorem 5.5 of Chen \cite{XC16},
	$a^{-3/2}\log\P\{u(t\,,0) > \e^a\} \to-d(1)$ as $a\to\infty$. Because
	\begin{equation}\label{gG}
		g_N \le \bm{1}_{[a_N,\infty)} \le G_N\qquad\text{for every $N>0$},
	\end{equation}
	it immediately follows from stationarity that, as $N\to\infty$,
	\[
		\E\int_0^N g_N(u(t\,,x))\,\d x  =N^{1-d(\alpha)+o(1)}
		\quad\text{and}\quad
		\E\int_0^N G_N(u(t\,,x))\,\d x =N^{1-d(\alpha)+o(1)}.
	\]
	On the other hand, both $g_N$ and $G_N$ are 1-Lipschitz. Therefore,
	\eqref{eq:PP} holds when $g=g_N$ as well as $g=G_N$. Because of this,
	Chebyshev's inequality ensures that for every fixed $\epsilon\in(0\,,1)$,
	\[
		\P\left\{ \left| \int_0^N g_N(u(t\,,x))\,\d x
		- \E\int_0^N g_N(u(t\,,x))\,\d x\right|
		>\epsilon \E\int_0^N g_N(u(t\,,x))\,\d x\right\}
		\le N^{-1+2d(\alpha)+o(1)},
	\]
	as $N\to\infty$. And the same estimate is valid when we replace $g_N$ everywhere
	by $G_N$. These facts and \eqref{gG} together show that, if
	$2d(\alpha)<1$, then \eqref{Dim} holds
	in probability. A standard subsequencing and blocking
	argument can  be used to prove a.s.-convergence in \eqref{Dim}.
	We skip this part, although we caution that
	some care is required in order to carry this out properly. This concludes the proof.
\end{proof}

\end{document}